\documentclass[12pt]{amsart}
\usepackage[utf8]{inputenc}
\usepackage[all]{xy}
\usepackage{amsmath, amsfonts, amssymb, amscd, amsthm}
\usepackage[margin=27mm,top=30mm]{geometry} 
\usepackage{graphicx}
\usepackage[colorinlistoftodos]{todonotes}
\usepackage[colorlinks=true, allcolors=blue]{hyperref}
\usepackage[tikz]{}
\usepackage{enumerate}
\usepackage{float}
\usepackage{textcomp}
\usepackage{color}
\usepackage{mathrsfs}

\newtheorem{theorem}{Theorem}[section]
\newtheorem{lemma}[theorem]{Lemma}
\newtheorem{corollary}[theorem]{Corollary}
\newtheorem{proposition}[theorem]{Proposition}
\newtheorem{conjecture}[theorem]{Conjecture}
\theoremstyle{definition}
\newtheorem{example}[theorem]{Example}
\newtheorem{definition}[theorem]{Definition}
\newtheorem{remark}[theorem]{Remark}

\newcommand{\forget}{\chi}
\newcommand{\ho}{Hom}

\DeclareMathOperator*{\mo}{mod}

\DeclareMathOperator*{\en}{End}
\DeclareMathOperator*{\si}{sign}

\numberwithin{equation}{section}

\title{The Coxeter Transformation on Cominuscule Posets}
\author{Emine Y\i ld\i r\i m}
\address{Department of Mathematics, Université du Québec à Montréal, Montréal, QC, H2X 3Y7.}
\email{yildirim.emine@courrier.uqam.ca }

\begin{document}

\begin{abstract}
Let $\mathsf{J(C)}$ be the poset of order ideals of a cominuscule poset $\mathsf{C}$ where $\mathsf{C}$ comes from two of the three infinite families of cominuscule posets or the exceptional cases. We show that the Auslander-Reiten translation $\tau$ on the Grothendieck group of the bounded derived category for the incidence algebra of the poset $\mathsf{J(C)}$, which is called the \emph{Coxeter transformation} in this context, has finite order. Specifically, we show that $\tau^{h+1}=\pm id$ where $h$ is the Coxeter number for the relevant root system.
\end{abstract}

\maketitle

\section*{Introduction}

Let $A$ be the incidence algebra of a poset $\mathsf{P}$ over a base field $\Bbbk$. If the poset $\mathsf{P}$ is finite, then $A$ is a finite dimensional algebra with finite global dimension. We are interested in incidence algebras coming from cominuscule posets. A cominuscule poset can be thought of as a parabolic analogue of the poset of positive roots of a finite root system. Cominuscule posets (also called minuscule posets) appear in the study of representation theory and algebraic geometry, especially in Lie theory and Schubert calculus~\cite{BL00}, \cite{Green13}. 

Let $\mathsf{J(C)}$ be the poset of order ideals of a given cominuscule poset $\mathsf{C}$. The poset $\mathsf{J(C)}$ is an interesting object in its own right. For instance, there is a correspondence between the elements of $\mathsf{J(C)}$ and the minimal coset representatives of the corresponding Weyl group~\cite{BL00}, \cite{RS13}. Many combinatorial properties of order ideals of cominuscule posets are explained by Thomas and Yong~\cite{TY09}. However, our main motivation for this paper comes from a conjecture by Chapoton which can be stated as follows.

Let $\mathsf{J(R)}$ be the poset of order ideals of $\mathsf{R}$ where $\mathsf{R}$ is the poset of positive roots of a finite root system $\Phi$. Let $H$ be a hereditary algebra of type $\Phi$ and let $\mathcal{T}_H$ be the poset of torsion classes. Consider the incidence algebras of $\mathsf{J(R)}$ and $\mathcal{T}_H$. Chapoton conjectures that there is a triangulated equivalence between the bounded derived categories $\mathcal{D}^b(\mo\mathsf{J(R)})$ and $\mathcal{D}^b(\mo\mathcal{T}_H)$ and that $\mathcal{D}^b(\mo\mathsf{J(R)})$ is fractionally Calabi-Yau, i.e. some non-zero power of the Auslander-Reiten translation $\tau$ equals some power of the shift functor. 

In~\cite{Chapoton07}, Chapoton proved that the Auslander-Reiten translation $\tau$ on the Grothendieck group of the bounded derived category (which is called \emph{Coxeter transformation} in this context) for Tamari posets is periodic. There is an abundance of studies on Coxeter transformation in the literature. We refer to Ladkani~\cite{L08} for references on the topic and recent results on the periodicity of Coxeter transformation. Kussin, Lenzing, Meltzer~\cite{KLM13} showed that the bounded derived category is fractionally Calabi-Yau for certain posets by using singularity theory. Recently, Diveris, Purin, Webb~\cite{DPP17} introduced a new method to determine whether the bounded derived category of a poset is fractionally Calabi-Yau.

In this paper, we consider Chapoton's Conjecture on the level of Grothendieck groups for cominuscule posets instead of root posets. We calculate the action of Auslander-Reiten translation $\tau$ on the bounded derived category $\mathcal{D}^b(\mo\mathsf{J(C)})$ of the incidence algebra of the poset $\mathsf{J(C)}$ of order ideals of a cominuscule poset $\mathsf{C}$, and then we write the action on the Grothendieck group. We show that the Auslander-Reiten translation $\tau$ acting on the corresponding Grothendieck group has finite order for two of the three infinite families of cominuscule posets, and for the exceptional cases. One can do this by considering the action of $\tau$ on any convenient set of generators. The obvious choices are the set of the isomorphism classes of simple modules, and the set of the isomorphism classes of projective modules. However, the periodicity of $\tau$ is not evident on these sets of generators. Instead, we consider a special spanning set of the Grothendieck group on which we see the periodicity combinatorially. 

The plan of the paper is as follows. In the first section, we give necessary background material we use throughout the paper. In the second section, we introduce a special collection of projective resolutions for \emph{grid posets}, and the corresponding combinatorics for the homology of these projective resolutions. These projective resolutions provide a spanning set for the Grothendieck group, and in the same section we show that the Coxeter transformation $\tau$ permutes this collection. The third section is devoted to setting up the combinatorial tools which we will use to prove the periodicity of $\tau$. In the fourth section, we prove the main result for grid posets. Finally, in the last section we extend our result to some of the cominuscule cases.

\subsection*{Acknowledgments} The author would like to thank her supervisor Hugh Thomas for his generous support and for his careful reading of this paper. The author was partially supported by ISM scholarships. The author also thanks Ralf Schiffler, Peter Webb and Nathan Williams for inspiring discussions. Finally, the author would like to thank the referee for pointing out Ladkani's \emph{flip-flop} technique (see Remark~\ref{referee}), and their careful reading of the paper.

\section{Preliminaries}

In this preliminary section, we will recall some basic definitions to fix notation in the paper.

\subsection{Basic Definitions}

We will use $\mathsf{P}$ to denote a poset. Two elements $a$, $b$ are comparable in $\mathsf{P}$ if $a\leq b$ or $b\leq a$, otherwise we say they are incomparable. We say $b$ covers $a$ in $\mathsf{P}$ if $a< b$ and there is no element $c$ such that $a< c< b$. An interval $[a,b]$ in $\mathsf{P}$ consists of all elements $x$ such that $a\leq x\leq b$. A poset $\mathsf{P}$ is called locally finite when every interval in $\mathsf{P}$ is finite. When $\mathsf{P}$ is locally finite, we will represent it by its \emph{Hasse diagram}, i.e. we represent every element in the poset $\mathsf{P}$ by a vertex and we put an arrow from a vertex $b$ to a vertex $a$ if $b$ covers $a$. We use the convention that the arrows in the Hasse diagram go downwards even though we simply draw edges in the figures.

A chain is a subset of $\mathsf{P}$ in which every pair of elements is comparable. An antichain is a subset of $\mathsf{P}$ such that every pair of different elements is incomparable.

\begin{definition} A grid poset $\mathsf{P_{m,n}}$ is the product of two chains of length $m$ and of length $n$. Explicitly, the elements of $\mathsf{P_{m,n}}$ are the pairs $(i,j)$ with $i\in \{1,\cdots,m\}$ and $j\in \{1,\cdots,n\}$. We compare elements \emph{entry-wise} in $\mathsf{P_{m,n}}$, i.e $(a,b)\leq (c,d)$ when $a\leq c$ and $b\leq d$. 
\end{definition}

A lattice $\mathsf{L}$ is a poset such that every pair of elements in $\mathsf{L}$ has a unique least upper bound, also called the \emph{join}, and a unique greatest lower bound, also called the \emph{meet}. A lattice is  distributive if the join and meet operations distribute over each other.

An \emph{order ideal} in $\mathsf{P}$ is a subset $\mathsf{I}\subseteq\mathsf{P}$ such that if $b\in \mathsf{I}$ and $a\leq b$, then $a\in \mathsf{I}$. We write $\mathsf{J(P)}$ for the set of all order ideals of $\mathsf{P}$. The partial order relation on $\mathsf{J(P)}$ is given by containment. Note that $\mathsf{J(P)}$ is always a distributive lattice. Recall that the fundamental theorem for finite distributive lattices states that if $\mathsf{L}$ is a finite distributive lattice, then there is a unique finite poset $\mathsf{P}$ such that $\mathsf{L}\cong \mathsf{J(P)}$ \cite[Theorem 3.4.1]{Stanley12}.

A non-decreasing finite sequence of positive integers $(a_1,\ldots,a_m)$ is called \emph{a partition}. We are going to represent such sequences by $\alpha=(\lambda_1^{\alpha_1},\cdots,\lambda_r^{\alpha_r})$ where $0\leq\lambda_1<\cdots<\lambda_r\leq n$ denote the distinct integers in the partition $\alpha$ while $\alpha_i$'s denote the number of repetitions of each $\lambda_i$.   

For each $k\in \{1,\ldots,m\}$, by the \emph{$k$-th row} of $\mathsf{P_{m,n}}$ we mean elements of the form $(k,b)$ in $\mathsf{P_{m,n}}$. We represent an order ideal $\mathsf{I}$ in $\mathsf{J(P_{m,n})}$ by its corresponding partition, i.e. the list of the number of elements of $\mathsf{I}$ in each row from top to bottom.

\begin{example}\label{poset1} Here is an example of a grid poset and its poset of order ideals with the corresponding partitions:
\begin{figure}[H]
\includegraphics[width=0.5\textwidth]{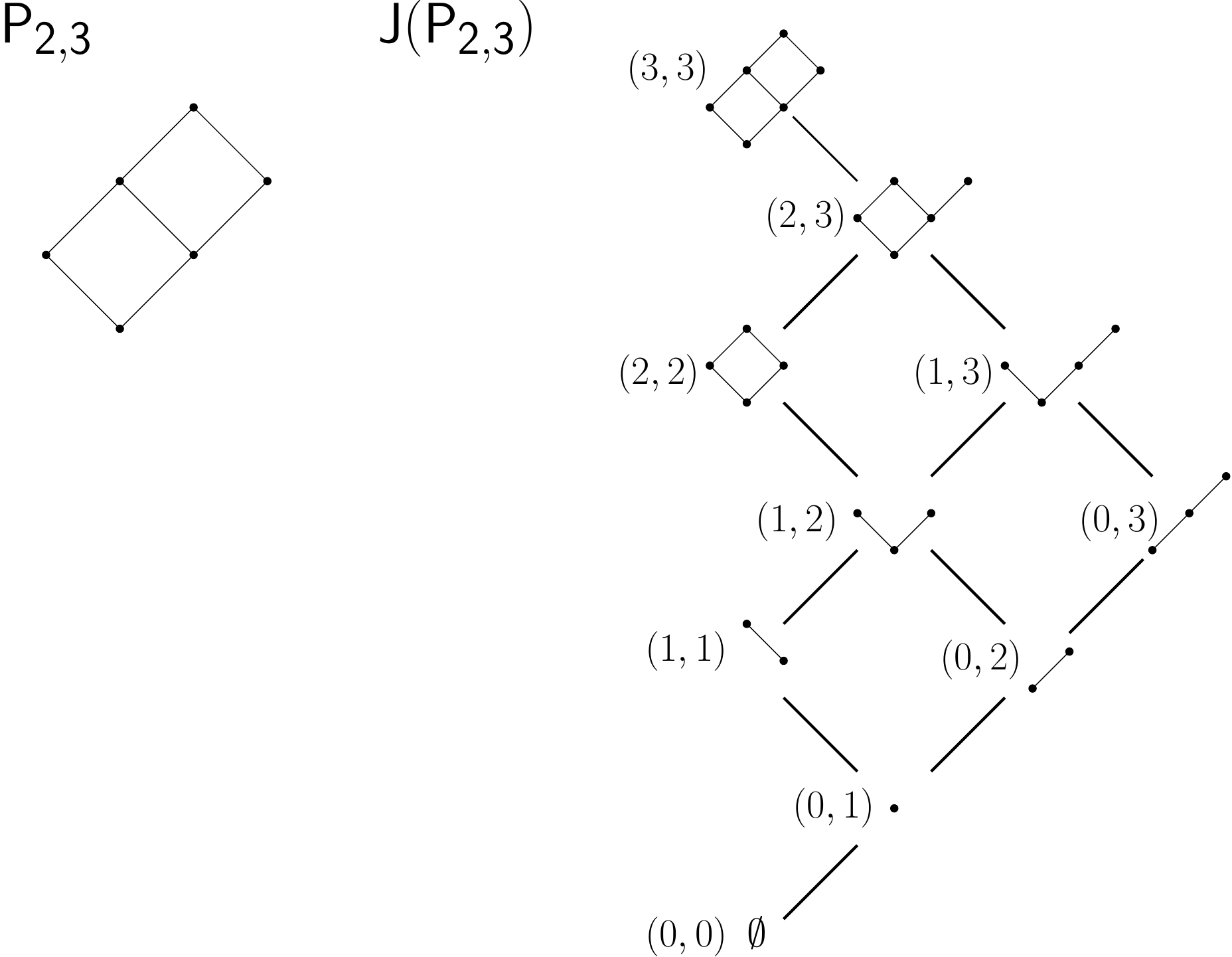}
\caption{The Hasse diagrams of the grid poset $\mathsf{P_{2,3}}$ and the poset of order ideals $\mathsf{J(P_{2,3})}.$}
\end{figure}
For instance, take the order ideal $\mathsf{I}$ represented by the partition $(2,3)$ in $\mathsf{J(P_{2,3})}$. The order ideal $\mathsf{I}$ consists of two rows: there are two elements in the first row, and we have three elements in the second row.
\end{example}

\subsection{Incidence algebra of $\mathsf{J(P_{m,n})}$} 

For a given locally finite poset $\mathsf{P}$, we define the incidence algebra of $\mathsf{P}$ as the path algebra of the Hasse diagram of $\mathsf{P}$ modulo the relation that any two paths are equal if their starting and ending points are the same. For two paths $p$ and $q$, we write the product as $pq$ if the starting point of $q$ equals to the ending point of $p$. Recall that we draw edges in the Hasse diagram oriented downwards.

Let $\mathcal{A}$ be the incidence algebra of $\mathsf{J(P_{m,n})}$. The algebra $\mathcal{A}$ has a finite global dimension since we do not have any oriented cycles in the quiver of $\mathsf{J(P_{m,n})}$. We will use $\mo\mathcal{A}$ to denote the category of finitely generated right modules over $\mathcal{A}$. Let us use $x_{\alpha\geq\beta}$ to denote the unique path that starts at vertex $\alpha$ and ends at vertex $\beta$. Let $P_{\alpha}=x_{\alpha\geq\alpha}\mathcal{A}$ be the indecomposable projective module over the algebra $\mathcal{A}$ for a vertex $\alpha$ in $\mathsf{J(P_{m,n})}$. The elements $x_{\alpha\geq\beta}$ form a $\Bbbk$-basis for $P_{\alpha}$. Then, the morphisms between two indecomposable projective modules $P_{\alpha}$ and $P_{\gamma}$ are \[\ho_{\mathcal{A}}(P_{\alpha},P_{\gamma})=\ho_{\mathcal{A}}(x_{\alpha\geq\alpha}\mathcal{A},x_{\gamma\geq\gamma}\mathcal{A})=x_{\gamma\geq\gamma}\mathcal{A}x_{\alpha\geq\alpha}\]
This is one dimensional if $\alpha\leq\gamma$, and $0$ otherwise. Similarly, we let $I_{\alpha}$ be an indecomposable injective module and $S_{\alpha}$ be a simple module corresponding to the vertex $\alpha$. For a more detailed description of indecomposable projective, injective and simple modules over a quotient of a path algebra, see \cite[Chapter 3]{ASS06}. 

\subsection{Derived categories}

Throughout this article, we will use the cohomological convention for complexes: all differentials increase the degree by one and we use superscripts to denote the degree at which the module is placed in the complex.

Let $\mathcal{B}$ be a finite-dimensional algebra. For a module $M$ from $\mo\mathcal{B}$, by the \emph{stalk complex of} $M$ we mean the complex which consists of just $M$ at one degree and $0$ everywhere else.

For a module $M$ from $\mo\mathcal{B}$, by a \emph{resolution} of $M$ we mean a complex $\mathcal{X} = (X^i)_{i\in\mathbb{Z}}$ where $X^i=0$ for all $i>0$ or for all $i<0$ with the property that the homology $H^0(\mathcal{X})$ is $M$ while $H^n(\mathcal{X})$ is zero for all $n\neq 0$. If $\mathcal{X}$ is bounded from above and each $X^i$ is projective, then $\mathcal{X}$ is called a \emph{projective resolution}. Similarly, if $\mathcal{X}$ is bounded from below and each $X^i$ is injective, then $\mathcal{X}$ is called an \emph{injective resolution}.

We will use $\mathcal{D}^b(\mo\mathcal{B})$ to denote the derived category of bounded complexes of $\mathcal{B}$-modules, and we simply write $\mathcal{D}^b(\mathcal{B})$.

\subsection{The Grothendieck Group} Our main reference for this subsection is~\cite[Chapter 3, Section 1]{Happel88}. 

The Grothendieck group $\mathcal{K}_0(\mathcal{B})$ of a finite-dimensional algebra $\mathcal{B}$ is an abelian group defined as a quotient of the free abelian group generated by isomorphism classes of objects in $\mo \mathcal{B}$ divided by the subgroup generated by the elements $[X]-[Y]+[Z]$ for every exact sequence $0\to X\to Y\to Z\to 0\text{ in } \mo \mathcal{B}.$ 

The Grothendieck group $\mathcal{K}_0(\mathcal{B})$ has a basis which consists of a representative of the isomorphism classes of each simple module. When $\mathcal{B}$ has finite global dimension, $\mathcal{K}_0(\mathcal{B})$ has also a basis coming from the representatives of the isomorphism classes of each indecomposable projective module~\cite[Chapter 3, Section 1.3]{Happel88}.

The Grothendieck group $\mathcal{K}_0$ of a triangulated category $\mathcal{C}$ is defined as a quotient of the free abelian group on the isomorphism classes of objects in $\mathcal{C}$ divided by the subgroup generated by the elements $[X]-[Y]+[Z]$ for every triangle $X\rightarrow Y\rightarrow Z\rightarrow X[1]$ of $\mathcal{C}$.

Let $\mathcal{K}_0(\mathcal{D}^b(\mathcal{B}))$ be the Grothendieck group of the bounded derived category of $\mo\mathcal{B}$ which is a triangulated category. There is an isomorphism from $\mathcal{K}_0(\mathcal{D}^b(\mathcal{B}))$ to $\mathcal{K}_0(\mathcal{B})$ given by the Euler characteristic of a complex~\cite[Chapter 3, Section 1.2]{Happel88}. 

\subsection{The Auslander-Reiten Translation} Given an object $X$ in $\mathcal{D}^b(\mathcal{B})$, the Auslander-Reiten translation $\tau(X)$ is defined as \[\tau(X)=(X\overset{L}{\otimes}_{\mathcal{B}}D\mathcal{B})[-1]\] where $D$ is the $\Bbbk$-duality functor, $[-1]$ is the shift functor, and $-\overset{L}{\otimes}_{\mathcal{B}}D\mathcal{B}$ is the left derived functor of the tensor product with $D\mathcal{B}$ over the algebra $\mathcal{B}$~\cite[Section 3.1]{Keller08}.

The functor $-\overset{L}{\otimes}_{\mathcal{B}}D\mathcal{B}$ has an easy description on the category of indecomposable projective modules of the algebra $\mathcal{B}$: it replaces a projective module $P_{\alpha}$ at a vertex $\alpha$ with the corresponding injective module $I_{\alpha}$ at the same vertex $\alpha$~\cite[Section 3.6]{Happel87}.

We calculate the Auslander-Reiten translation $\tau$ for a stalk complex $\mathcal{X}$ as follows: we replace $\mathcal{X}$ by its projective resolution $\mathcal{P}$ and then apply the functor $-\otimes_{\mathcal{B}}D\mathcal{B}$ to each term of $\mathcal{P}$, and finally shift the resulting injective complex by one to the right.

Since $\tau$ is an auto-equivalence of the triangulated category, after applying $\mathcal{K}_0$ we get an automorphism of the corresponding Grothendieck groups. Thus, $\tau$ defines a bijective linear map on $\mathcal{K}_0(\mathcal{B})$~\cite[Chapter 3, Section 1]{Happel88}.  This map is also called the Coxeter transformation. 

\begin{example} We continue in the setting of Example \ref{poset1}. We take the incidence algebra $\mathcal{A}$ of the poset $\mathsf{J(P_{2,3})}$. Let $S_{(1,2)}$ be the simple module over the vertex $(1,2)$. Then, consider the complex $S: 0\rightarrow S^0_{(1,2)}\rightarrow 0$ in  $\mathcal{D}^b(\mathcal{A})$. Now, we calculate $\tau S$ as follows:
\begin{enumerate}
\item Replace $S$ by its projective resolution $$P: 0\rightarrow P^{-2}_{(0,1)}\rightarrow P^{-1}_{(1,1)}\oplus P^{-1}_{(0,2)}\rightarrow P^0_{(1,2)}\rightarrow 0$$
\item Apply the functor $-\otimes_{\mathcal{A}}D\mathcal{A}$ to each term of the projective resolution $P$, then each projective module is replaced by the corresponding injective module $$I: 0\rightarrow I^{-2}_{(0,1)}\rightarrow I^{-1}_{(1,1)}\oplus I^{-1}_{(0,2)}\rightarrow I^0_{(1,2)}\rightarrow 0$$
Here the complex $I$ is quasi-isomorphic to the complex $0\rightarrow S^{-2}_{(0,1)}\rightarrow 0$.
\item Finally we apply the shift and get $$S': 0\rightarrow S^{-1}_{(0,1)}\rightarrow 0.$$
\end{enumerate}
Thus, $\tau S= S'$. On the level of Grothendieck group we have $\tau ([S_{(1,2)}])=-[S_{(0,1)}]$.
\end{example}

\section{Modules and Intervals}

\subsection{Projective resolutions}\label{s-proj-res} In this subsection, we will describe a special collection of projective resolutions in $\mathcal{D}^b(\mo\mathcal{A})$ that will span the Grothendieck group. In order to prove the periodicity of $\tau$, we are going to need these resolutions.

\begin{definition} We call a tuple of the following form an \emph{enhanced partition}: $$(0^{\alpha_0}|\lambda_1^{\alpha_1},\cdots,\lambda_r^{\alpha_r}|n^{\alpha_{r+1}})$$ where $0\leq\lambda_1<\lambda_2<\cdots<\lambda_r\leq n$ and we will refer to the $0$'s and the $n$'s separated by bars as fixed entries. Note that $\alpha_0$ and $\alpha_{r+1}$ can be zero. 
\end{definition} 
Let $E$ be the set of enhanced partitions, and let $F$ be the set of partitions. We now define a function $\forget$ from $E$ to $F$, which sends an enhanced partition to a partition by forgetting the bars. This means there are no fixed entries anymore. Formally,
\begin{equation}\label{forget}
\forget[(0^{\alpha_0}|\lambda_1^{\alpha_1},\cdots,\lambda_r^{\alpha_r}|n^{\alpha_{r+1}})]=(0^{\alpha_0},\lambda_1^{\alpha_1},\cdots,\lambda_r^{\alpha_r},n^{\alpha_{r+1}})
\end{equation}
The function $\forget$ allows us to treat enhanced partitions as usual partitions. We need enhanced partitions to write a special collection of resolutions.  

Let $E_{L}$ be the set of enhanced partitions of the form $\alpha=(0^{\alpha_0}|\lambda_1^{\alpha_1},\lambda_2^{\alpha_2},\cdots,\lambda_r^{\alpha_r}|n^{\alpha_{r+1}})$ where $\lambda_1\neq 0$ and $E_{R}$ be to the set of enhanced partitions of the form $\alpha=(0^{\alpha_0}|\lambda_1^{\alpha_1},\lambda_2^{\alpha_2},\cdots,\lambda_r^{\alpha_r}|n^{\alpha_{r+1}})$ where $\lambda_r\neq n$.

\begin{definition} For a given enhanced partition $\alpha=(0^{\alpha_0}|\lambda_1^{\alpha_1},\cdots,\lambda_r^{\alpha_r}|n^{\alpha_{r+1}})$, let $R_{\alpha}$ be the set of indices of the nonzero entries $\lambda_i$ in $\alpha$. 
\end{definition}

If $\alpha\in E_L$, then $R_{\alpha}=\{1,2,\cdots,r\}$, and if $\alpha\notin E_L$, then $R_{\alpha}=\{2,\cdots,r\}$. 

Let $\delta_i$ be the sequence of $0$'s and $1$'s where $1$ is placed at those places $\lambda_i$'s appear in $\alpha$. Denote also $\delta_X=\sum_{i\in X}\delta_i$ for a subset $X\subseteq R_{\alpha}$. Notice that we define $\delta_i$'s only for non-fixed and nonzero entries in $\alpha.$ 

\begin{definition} For an enhanced partition $\alpha=(0^{\alpha_0}|\lambda_1^{\alpha_1},\cdots,\lambda_r^{\alpha_r}|n^{\alpha_{r+1}})\in E_L$, we define a complex of projective modules $\mathcal{P}_{\alpha}$ as follows: 

\begin{equation}~\label{Palpha}
\mathcal{P}_{\alpha}:0\rightarrow P^{-r}_{\alpha-\delta_{R_{\alpha}}}\xrightarrow{\partial^{-r}} \bigoplus_{\substack{J\subseteq R_{\alpha},\\|J|=r-1}} P^{-r+1}_{\alpha-\delta_J}\rightarrow \cdots\rightarrow \bigoplus_{\substack{J\subseteq R_{\alpha},\\|J|=1}} P^{-1}_{\alpha-\delta_J}\xrightarrow{\partial^{-1}} P^0_{\alpha}\rightarrow 0
\end{equation} with the maps 

$$\partial^{-k}: \bigoplus_{\substack{J\subseteq R_{\alpha},\\|J|=k}} P_{\alpha-\delta_J}\to \bigoplus_{\substack{J\subseteq R_{\alpha},\\|J|=k-1}} P_{\alpha-\delta_J}, \qquad x_{(\alpha-\delta_{\{i_1,\cdots,i_k\}}\geq\beta)}\mapsto \sum_t(-1)^tx_{(\alpha-\delta_{\{i_1,\cdots,\widehat{i_t},\cdots,i_k\}}\geq\beta)}$$ for each $J=\{i_1,\cdots,i_k\}\subseteq R_{\alpha}$ and $x_{\alpha-\delta_J\geq\beta}\in P_{\alpha-\delta_J}.$ 
\end{definition}

\begin{remark} The grading of the complex $\mathcal{P}_{\alpha}$ comes from the cardinality of $J\subseteq R_{\alpha}$, and $\alpha-\delta_J$ is just a vector subtraction.
\end{remark}

\begin{proposition}\label{pr} The complex $\mathcal{P}_{\alpha}$ in Equation~\eqref{Palpha} defines a projective resolution.
\end{proposition}

\begin{proof} Notice that for any $\beta\leq\beta'$ there is a unique embedding of $P_{\beta}$ into $P_{\beta'}$ by left multiplication with $x_{\beta'\geq\beta}$ sending $x_{\beta\geq\gamma}\mapsto x_{\beta'\geq\gamma}$ for each $\gamma\leq\beta$. Thus, the maps $\partial^{-k}$ are all right $\mathcal{A}$-module maps. Therefore, it is enough to show that we have a complex in the category of $\Bbbk$-vector spaces.

For every $\beta\leq\alpha$, the graded $\Bbbk$-subspace $\mathcal{P}_{\alpha}\cdot x_{\beta\geq\beta}$ is actually a $\Bbbk$-subcomplex of $\mathcal{P}_{\alpha}$ since the differentials preserve the grading. Therefore, it is enough to prove the exactness of $\mathcal{P}_{\alpha}$ by showing we have an exact complex at each vertex in the support of $P_{\alpha}$. Note that $P_{\alpha}$ has support over the vertices $\beta\leq\alpha$.

For a given $\beta\leq\alpha$ we find the maximal $J_{\beta}\subseteq R_{\alpha}$ so that the inequality $\beta\leq \alpha-\delta_{J_{\beta}}$ holds. Let $k=|J_{\beta}|$. When we multiply the complex by $x_{\beta\geq\beta}$ on the right, each of the projective modules in $\mathcal{P}_{\alpha}$ reduces to the ground field $\Bbbk$. Moreover, after the reduction, the face maps in the differentials are all identity maps. Now, $J_{\beta}$ determines which summands of $\mathcal{P}_{\alpha}$ have $S_{\beta}$ in their composition series. Then we have the following subcomplex of $\Bbbk$-vector spaces:

$$\mathcal{S} : 0\rightarrow \Bbbk^{-k}\rightarrow\bigoplus_{\binom{k}{k-1}}\Bbbk^{-k+1}\rightarrow\cdots\rightarrow\bigoplus_{\binom{k}{2}}\Bbbk^{-2}\rightarrow\bigoplus_{\binom{k}{1}}\Bbbk^{-1}\rightarrow \Bbbk^0\rightarrow 0$$ with the differential as defined above. This is the face complex of the standard $(k-1)$-simplex. Since the standard $(k-1)$-simplex is contractible, its reduced homology is zero provided $k-1\geq 0$~\cite[Section 2.1]{Hatcher02}. This means that there is a homology if and only if $k=|J_{\beta}|=0$.
\[ H^{n}(\mathcal{S}) =
  \begin{cases}
    0       & \quad \text{if } 0<k\leq r, n\neq 0,\\
    \Bbbk  & \quad \text{if } k=0,\ n=0.\\
  \end{cases}
\] Then, we have
\[ H^{n}(\mathcal{P}_{\alpha})\cdot x_{\beta\geq\beta} =
  \begin{cases}
    0       & \quad \text{if } 0<k\leq r, n\neq 0,\\
    \Bbbk & \quad \text{if } k=0,\ n=0.\\
  \end{cases}
\] and this implies that $\mathcal{P}_{\alpha}$ is a projective resolution. 
\end{proof}

\begin{remark}\label{remark1} Notice that the homology of $\mathcal{P}_{\alpha}$ is supported only over vertices $\beta$ such that $J_{\beta}=\emptyset$. This implies that $\beta\leq \alpha$ and $\beta\nleq \alpha-\delta_i$ for each $i\in R_{\alpha}$. We will further investigate the homology of $\mathcal{P}_{\alpha}$ in Subsection~\ref{modules-intervals}.
\end{remark}

\begin{example}\label{pr-re} Let $\mathcal{A}$ be the incidence algebra for the poset $\mathsf{J(P_{5,7})}$. Let us consider $\alpha=(0|2,2,3,7|)$. Then $R_{\alpha}=\{1,2,3\}$ and 

\begin{align*}
\mathcal{P}_{\alpha} : 0\rightarrow P^{-3}_{(0,1,1,2,6)}\rightarrow & P^{-2}_{(0,1,1,2,7)}\oplus P^{-2}_{(0,2,2,2,6)}\oplus P^{-2}_{(0,1,1,3,6)}\rightarrow \\
& P^{-1}_{(0,1,1,3,7)}\oplus P^{-1}_{(0,2,2,2,7)}\oplus P^{-1}_{(0,2,2,3,6)} \rightarrow P^0_{(0,2,2,3,7)}\rightarrow 0.
\end{align*}

Now let $\alpha'=(0|2,2,3|7)$. Note that $\alpha'$ is the same as $\alpha$ except that $7$ is now a fixed entry. Then we have $R_{\alpha'}=\{1,2\}$ and 
$$\mathcal{P}_{\alpha'} : 0\rightarrow P^{-2}_{(0,1,1,2,7)}\rightarrow P^{-1}_{(0,1,1,3,7)}\oplus P^{-1}_{(0,2,2,2,7)}\rightarrow P^0_{(0,2,2,3,7)}\rightarrow 0$$ 
\end{example}

\subsection{Action of the Auslander-Reiten translation on the projective resolutions}\label{tau-inj-res}
In this subsection, we are going to look at the action of the Auslander-Reiten translation on the projective resolutions $\mathcal{P}_{\alpha}$ and discuss the homology of the resulting complex.

\begin{proposition}\label{ir}
Let $\mathcal{P}_{\alpha}$ be the projective resolution defined in Equation~\eqref{Palpha}. If we apply the Auslander-Reiten translation $\tau$ to the complex $\mathcal{P}_{\alpha}$, the resulting complex is an injective resolution up to a shift.
\end{proposition}

\begin{proof}
After applying $-\otimes D\mathcal{A}$ on $\mathcal{P}_{\alpha}$, we get the following injective complex:
\begin{equation}\label{Ialpha}
\mathcal{I}_{\alpha}:0\rightarrow I^{-r}_{\alpha-\delta_{R_{\alpha}}}\rightarrow \displaystyle\bigoplus_{|J|=r-1} I^{-r+1}_{\alpha-\delta_J}\rightarrow \cdots\rightarrow \displaystyle\bigoplus_{|J|=1} I^{-1}_{\alpha-\delta_J}\rightarrow I^0_{\alpha}\rightarrow 0
\end{equation}

The proof of Proposition~\ref{pr} with some modifications can be applied here. Firstly, notice that $I_{\alpha-\delta_J}$ has support over the vertices $\gamma\geq\alpha-\delta_J$. Then we write the subcomplexes as follows: for a given $\gamma$ we find the minimal $J_{\gamma}$ such that the inequality $\gamma\geq\alpha-\delta_{J_{\gamma}}$ holds. Let $k=|J_{\gamma}|$, then the vertices $\gamma$ will appear as in the following:

$$0\rightarrow \Bbbk^{-r}_{\gamma}\rightarrow\bigoplus_{\binom{r-k}{1}}\Bbbk^{-r+1}_{\gamma}\rightarrow\cdots\rightarrow\bigoplus_{\binom{r-k}{r-k-2}}\Bbbk^{-k-2}_{\gamma}\rightarrow\bigoplus_{\binom{r-k}{r-k-1}}\Bbbk^{-k-1}_{\gamma}\rightarrow \Bbbk^{-k}_{\gamma}\rightarrow 0$$ This is another face complex which only has homology when $k=r$. Consequently, the complex~\eqref{Ialpha} is an injective resolution up to a shift. 

\end{proof}

\begin{remark}\label{remark2}
As in the projective case, the homology of $\mathcal{I}_{\alpha}$ has support over vertices $\gamma$ only when $J_{\gamma}=R_{\alpha}$. This means that $\alpha-\delta_{R_{\alpha}}\leq\gamma$ and $\alpha-\delta_J\nleq\gamma$ where $|J|=r-1$. We will further investigate the homology of $\mathcal{I}_{\alpha}$ in Subsection~\ref{modules-intervals}.
\end{remark}

\subsection{Intervals in the poset $\mathsf{P_{m,n}}$}\label{tildef} In this subsection, we are going to define two functions: $f$ from $E_{L}$ to $E_{R}$, and $g$ from $E_{R}$ to $E_{L}$. These functions will help us to describe the homology of the complexes defined in Subsections~\ref{s-proj-res} and \ref{tau-inj-res} combinatorially. 

Let $\alpha=(0^{\alpha_0}|\lambda_1^{\alpha_1},\lambda_2^{\alpha_2},\cdots,\lambda_r^{\alpha_r}|n^{\alpha_{r+1}})\in E_L$ and then $R_{\alpha}=\{1,2,\cdots,r\}$. 
  
\begin{enumerate}
\item Let $\alpha\in E_{L}$. The function $f(\alpha)$ is defined as follows: We first apply $\forget$ which is defined in \eqref{forget} to $\alpha$. Then for each $i\in R_{\alpha}$, we leave the last occurrence of $\lambda_i$ in $\alpha$ unchanged while we minimize the rest of the occurrences including the fixed $n$'s at the end if there are any, thus making the partition as small as possible. 
Finally, we enhance the result as follows: the first bar is placed in the same position as in $\alpha$ and if there is no $n$ in $f(\alpha)$ the second bar obviously goes at the end, while if there are $n$'s we put the second bar before all of them.

Formally,
\begin{align*}
f((0^{\alpha_0}|\lambda_1^{\alpha_1},&\lambda_2^{\alpha_2},\cdots,\lambda_r^{\alpha_r}|n^{\alpha_{r+1}}))\\
    =&\begin{cases}
    (0^{\alpha_0}|0^{\alpha_1-1},\lambda_1^{\alpha_2},\lambda_2^{\alpha_3},\cdots,\lambda_{r-1}^{\alpha_r}|\lambda_r^{1+\alpha_{r+1}})&\text{ if } \lambda_r= n,\\
    (0^{\alpha_0}|0^{\alpha_1-1},\lambda_1^{\alpha_2},\lambda_2^{\alpha_3},\cdots,\lambda_{r-1}^{\alpha_r},\lambda_r^{1+\alpha_{r+1}}|)&\text{ if } \lambda_r\neq n.
    \end{cases}
\end{align*}

\item Let $\alpha=(0^{\alpha_0}|\lambda_1^{\alpha_1},\lambda_2^{\alpha_2},\cdots,\lambda_r^{\alpha_r}|n^{\alpha_{r+1}})\in E_R$. Note that here $R_{\alpha}=\{1,2,\cdots,r\}$ if $\lambda_1\neq 0$ and $R_{\alpha}=\{2,\cdots,r\}$ if $\lambda_1=0$.
The function $g(\alpha)$ is defined as follows: We first apply $\forget$ to $\alpha$. For each $i\in R_{\alpha}$, we leave the first occurrence of $\lambda_i$ in $\alpha$ unchanged while maximizing the rest of the occurrences, thus making the partition as large as possible. Notice that we do not change $0$'s which were fixed in $\alpha$, but we do maximize the unfixed $0$'s. 
Then we place the first bar in the same place as in $\alpha$; the position of second bar can be seen in the following formal definition.

Formally,
\begin{align*}
g((0^{\alpha_0}|\lambda_1^{\alpha_1},&\lambda_2^{\alpha_2},\cdots,\lambda_r^{\alpha_r}|n^{\alpha_{r+1}}))\\
    =& \begin{cases}
       (0^{\alpha_0}|\lambda_2^{\alpha_1+1},\lambda_3^{\alpha_2},\cdots,\lambda_r^{\alpha_{r-1}}|n^{\alpha_{r}-1}) &\text{ if } \alpha_{r+1}=0\text{ and }\lambda_1=0,\\
       (0^{\alpha_0}|\lambda_2^{\alpha_1+1},\lambda_3^{\alpha_2},\cdots,\lambda_r^{\alpha_{r-1}},n^{\alpha_{r}}|n^{\alpha_{r+1}-1}) &\text{ if } \alpha_{r+1}\neq 0\text{ and }\lambda_1=0,\\
       (0^{\alpha_0}|\lambda_1,\lambda_2^{\alpha_1},\lambda_3^{\alpha_2},\cdots,\lambda_r^{\alpha_{r-1}}|n^{\alpha_{r}-1}) &\text{ if } \alpha_{r+1}=0 \text{ and } \lambda_1\neq 0,\\
       (0^{\alpha_0}|\lambda_1,\lambda_2^{\alpha_1},\lambda_3^{\alpha_2},\cdots,\lambda_r^{\alpha_{r-1}},n^{\alpha_{r}}|n^{\alpha_{r+1}-1})&\text{ if }\alpha_{r+1}\neq 0\text{ and }\lambda_1\neq 0.
     \end{cases}
\end{align*}
\end{enumerate}

\begin{example}\label{functions-on-enhanced} Let $\alpha=(0^2|2^3,5,6^4,9^2|13^2)=(0,0|2,2,2,5,6,6,6,6,9,9|13,13)$. Then $R_{\alpha}=\{1,2,3,4\}$. To find $f(\alpha)$ we first apply $\forget$. Now, we fix the last occurrences of each $\lambda_i$, $i\in R_{\alpha}$ while minimizing the rest as shown in the following:

\begin{align*}
& (0,0,2,2,2,5,6,6,6,6,9,9,13,13)\\
\text{We leave the last occurrences of each } \lambda_i \text{ unchanged: } & (0,0,2,2,\textcolor{red}{2}, \textcolor{red}{5},6,6,6,\textcolor{red}{6},9,\textcolor{red}{9},13,13)\\
\text{We minimize: } &(0,0,0,0,\textcolor{red}{2},\textcolor{red}{5},5,5,5,\textcolor{red}{6},6,\textcolor{red}{9},9,9)\\
\text{Then we enhance: } &(0,0|0,0,2,5,5,5,5,6,6,9,9,9|)
\end{align*}

The result is $f(\alpha)=(0^2|0^2,2,5^4,6^2,9^3|)$. 

Similarly, we can get $g(\alpha)=g((0^2|2^3,5,6^4,9^2|13^2))=(0^2|2,5^3,6,9^4,13^2|13).$
\end{example}

\begin{lemma}\label{fog} The functions $f$ and $g$ are inverses of each other.
\end{lemma}

\begin{proof} Let $\alpha=(0^{\alpha_0}|\lambda_1^{\alpha_1},\lambda_2^{\alpha_2},\cdots,\lambda_r^{\alpha_r}|n^{\alpha_{r+1}})\in E_{L}$. Then
\begin{align*}
g\circ f(\alpha)=\begin{cases}
 g((0^{\alpha_0}|0^{\alpha_1-1},\lambda_1^{\alpha_2},\lambda_2^{\alpha_3},\cdots,\lambda_{r-1}^{\alpha_r}|\lambda_r^{\alpha_{r+1}+1}))&\text{ if } \lambda_r=n\\
 g((0^{\alpha_0}|0^{\alpha_1-1},\lambda_1^{\alpha_2},\lambda_2^{\alpha_3},\cdots,\lambda_{r-1}^{\alpha_r},\lambda_r^{\alpha_{r+1}+1}|))&\text{ if } \lambda_r\neq n\\
\end{cases}
\end{align*}
We can easily conclude that the result is $\alpha$. Similarly, one can show that $f\circ g(\beta)=\beta$ for $\beta\in E_{R}$.
\end{proof}

\subsection{Homologies and Intervals}\label{modules-intervals}
In this subsection, we will discuss the homology of $\mathcal{P}_{\alpha}$, and the homology of $\mathcal{I}_{\alpha}$ in relation to intervals in the poset $\mathsf{J(P_{m,n})}$.

\begin{definition} Let $[\gamma,\gamma']$ be an interval in $\mathsf{J(P_{m,n})}$. We define the corresponding element in the Grothendieck group $\mathcal{K}_0$ for the interval $[\gamma,\gamma']$ as $[[\gamma,\gamma']]:={\displaystyle \sum_{\gamma\leq x\leq \gamma'}}[S_x]$.
\end{definition}

\begin{proposition}\label{proj-int}
The class of the projective resolution $\mathcal{P}_{\alpha}$ in $\mathcal{K}_0$ is $[[f(\alpha),\alpha]]$ for every $\alpha\in E_L$.
\end{proposition}

\begin{proof} Let $I=\{x\in\mathsf{P_{m,n}}\mid x\leq \alpha \text{ and } x\nleq \alpha-\delta_i \text{ for any } i\in R_{\alpha}\}$. Firstly, notice that the class in $\mathcal{K}_0$ for the homology of the projective resolution $\mathcal{P}_{\alpha}$ is supported over the vertices in $I$ by the result in Proposition~\ref{pr}. Recall also Remark~\ref{remark1}. We will prove that $I$ forms an interval in the poset $\mathsf{J(P_{m,n})}$. Clearly, $\alpha$ is the maximum element in $I$. 

Let $\alpha=(0^{\alpha_0}|\lambda_1^{\alpha_1},\lambda_2^{\alpha_2},\cdots,\lambda_r^{\alpha_r}|n^{\alpha_{r+1}})$. We also write 
$\alpha=(a_1,a_2,\cdots,a_m)$. Then, we write $\alpha-\delta_i=(b_1,b_2,\cdots,b_m)$ where 
\begin{align*}  b_j =
\begin{cases}
 a_j  \qquad &\text{    when } a_j\neq \lambda_i,\\
 a_j-1 \qquad &\text{   otherwise.}
\end{cases}
\end{align*} for $1\leq j\leq m$. Let $x=(c_1,c_2,\cdots,c_m)\in I$. In order for $x\leq\alpha$ but $x\nleq\alpha-\delta_i$ for all $i\in R_{\alpha}$, for each $i$ we must have $c_i\leq a_i$, and at least one of $c_{\alpha_0+\cdots+\alpha_{i-1}+1},\cdots,c_{\alpha_0+\cdots+\alpha_i}$ must be greater than $b_{\alpha_0+\cdots+\alpha_{i-1}+1},\cdots,b_{\alpha_0+\cdots+\alpha_i}$, i.e. must equal $\lambda_i$. Since $c_1\leq c_2\leq\cdots\leq c_m$, it must be that $c_{\alpha_0+\cdots+\alpha_i}=\lambda_i$. Now, it is clear that $$I=\{x\in\mathsf{P_{m,n}}\mid c_{\alpha_0+\cdots+\alpha_i}=\lambda_i \text{ for each } i\in R_{\alpha}\}.$$ This is an interval having minimum $f(\alpha)$. This completes the proof. 
\end{proof}

The following is an illustration of the proof with an example. For this example, let us assume $\alpha=(|2,2,6,6,6,6,9,9,9,9,9|)$.  We illustrate the corresponding order ideal with the black contour in Figure~\ref{funcfunc}. Then we have $R_{\alpha}=\{1,2,3\}$. For instance, $\alpha-\delta_2=(2,2,5,5,5,5,9,9,9,9,9)$ which is shown with the red contour in Figure~\ref{funcfunc}. The gray box shows the row where $c_{\alpha_0+\alpha_1+\alpha_2}=c_6=\lambda_2=6$. Finally, $f(\alpha)$ is illustrated with the blue dotted contour in Figure~\ref{funcfunc}.

\begin{figure}[H]
\includegraphics[width=0.37\textwidth]{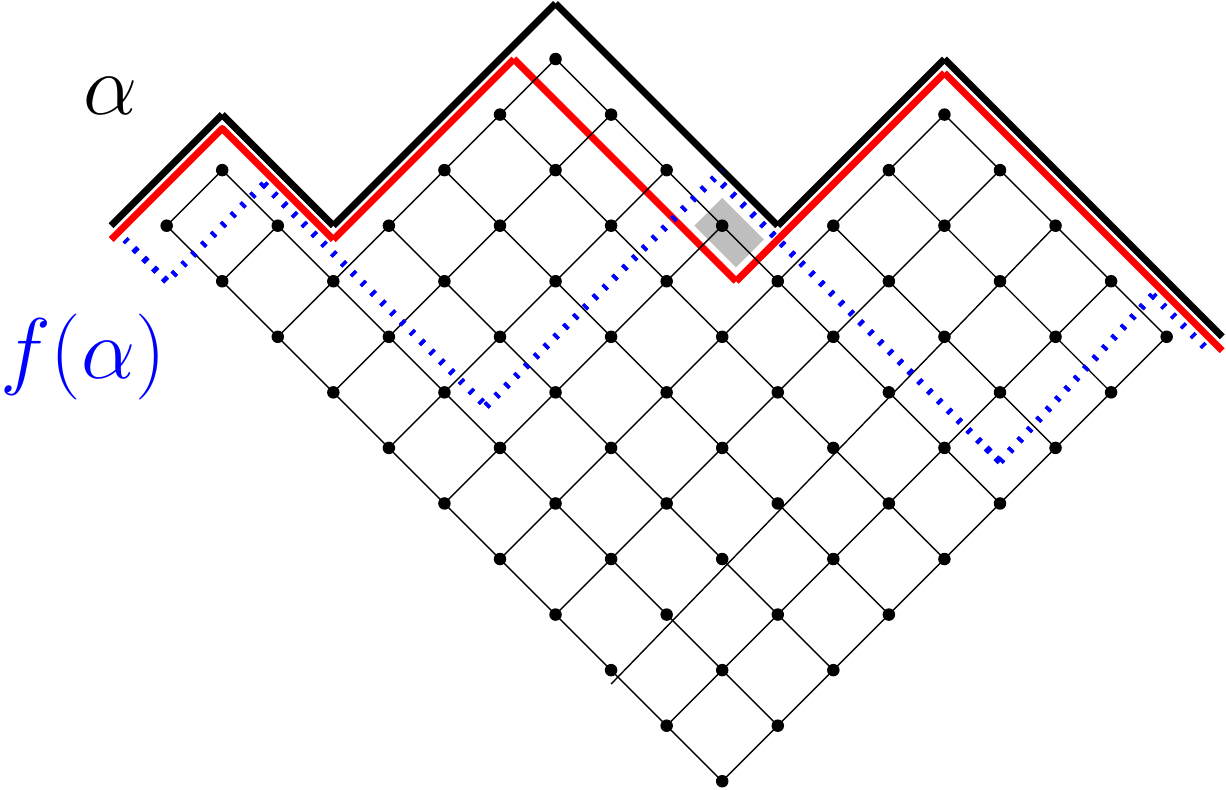}
\caption{An illustration of the function $f(\alpha)$}\label{funcfunc}
\end{figure}

\begin{example} As in Example \ref{pr-re}, let $\alpha=(0|2,2,3,7|)$ and consider its projective resolution. The corresponding element in $\mathcal{K}_0$ for the homology of this projective resolution is $[[f(\alpha),\alpha]]=[(0,0,2,3,7),(0,2,2,3,7)]$. Now, assume $\alpha'=(0|2,2,3|7)$, then $[[f(\alpha'),\alpha']]=[(0,0,2,3,3),(0,2,2,3,7)].$
\end{example}

In the following, we would like to analyze the homology of injective resolution $\mathcal{I}_{\alpha}$ after the action of $\tau$ on the projective resolution $\mathcal{P}_{\alpha}$. 

\begin{proposition}\label{inj-int}
The class of the injective resolution $\mathcal{I}_{\alpha}$ in $\mathcal{K}_0$ is $\pm [[\alpha-\delta_{R_{\alpha}},g(\alpha-\delta_{R_{\alpha}})]]$ for every $\alpha\in E_L$.
\end{proposition}

\begin{remark}\label{alpha-delta}
Before proving Proposition~\ref{inj-int}, we need to discuss the rule that enhances the partition $\alpha-\delta_{R_{\alpha}}$ so that we can apply the function $g$. Let us assume $\alpha=(0^{\alpha_0}|\lambda_1^{\alpha_1},\lambda_2^{\alpha_2},\cdots,\lambda_r^{\alpha_r}|n^{\alpha_{r+1}})\in E_L$. Then the enhanced partition $\alpha-\delta_{R_{\alpha}}$ is defined as follows: The position of the second bar is the same as in $\alpha$. This implies that we fix all of the $n$'s in $\alpha-\delta_{R_{\alpha}}$. Now, if we have $0$'s in $\alpha-\delta_{R_{\alpha}}$, we have to determine which of them are fixed. To do so, we will look at $\alpha$. Recall that $\lambda_1\neq 0$ in $\alpha$ for $\mathcal{P}_{\alpha}$. If $\lambda_1\neq 1$, then we do not fix any $0$'s in $\alpha-\delta_{R_{\alpha}}$, i.e we put the first bar before all of the entries. If $\lambda_1=1$, then we look at the location of first appearance of $\lambda_1$ in $\alpha$, say in $k$-th position from the beginning. Then we put the first bar after the $k$-th $0$ in $\alpha-\delta_{R_{\alpha}}$. Formally,
\begin{align*}
\alpha-\delta_{R_{\alpha}}=\begin{cases}
(|0^{\alpha_0},(\lambda_1-1)^{\alpha_1},(\lambda_2-1)^{\alpha_2},\cdots,(\lambda_r-1)^{\alpha_r}|n^{\alpha_{r+1}}) &\text{ if } \lambda_1\neq 1\\
(0^{\alpha_0+1}|0^{\alpha_1-1},(\lambda_2-1)^{\alpha_2},\cdots,(\lambda_r-1)^{\alpha_r}|n^{\alpha_{r+1}}) &\text{ if } \lambda_1=1.
\end{cases}
\end{align*}
\end{remark}

\begin{proof}[Proof of Proposition~\ref{inj-int}]
Here finding $g(\alpha-\delta_{R_{\alpha}})$ is the dual of finding $f(\alpha)$. 

Let $I'=\{x\in\mathsf{P_{m,n}}\mid \alpha-\delta_{R_{\alpha}}\leq x \text{ and } x\ngeq\alpha-\delta_J \text{ with } |J|=r-1\}$. In this case, we know the minimum element of $I'$ is $\alpha-\delta_{R_{\alpha}}$. 

Let $\alpha=(0^{\alpha_0}|\lambda_1^{\alpha_1},\lambda_2^{\alpha_2},\cdots,\lambda_r^{\alpha_r}|n^{\alpha_{r+1}})$, and let $x=(c_1,c_2,\cdots,c_m)\in I'$. We can deduce the following as in the proof of Proposition~\ref{proj-int}: In order for $\alpha-\delta_{R_{\alpha}}\leq x$ and  $x\ngeq\alpha-\delta_J$ where $|J|=r-1$, the entries $c_{\alpha_0+\cdots+\alpha_{i-1}+1}$ must equal $\lambda_i-1$ for each $i$. Now, we can conclude that $$I'=\{x\in\mathsf{P_{m,n}}\mid c_{\alpha_0+\cdots+\alpha_{i-1}+1}=\lambda_i-1 \text{ for each } i\in R_{\alpha}\}.$$ This is an interval with the maximum $g(\alpha-\delta_{R_{\alpha}})$. This finishes the proof.
\end{proof} 

We will illustrate the idea of the proof by an example as we study in the previous case. Assume $n=9$ and $\alpha=(|2,2,6,6,6,6,9,9,9,9,9|)$. Then $R_{\alpha}=\{1,2,3\}$ and $\alpha-\delta_{R_{\alpha}}=(|1,1,5,5,5,5,8,8,8,8,8|)$.  The corresponding order ideal is illustrated with the black contour in Figure~\ref{guncgunc}. Also, we can calculate that $g(\alpha-\delta_{R_{\alpha}})=(|1,5,5,8,8,8,8|9,9,9,9)$ as illustrated with the blue contour. The red contour shows $\alpha-\delta_{\{1,2\}}=(1,1,5,5,5,5,9,9,9,9,9)$. The gray box shows the row $\alpha_0+\alpha_1+\alpha_2+1=7$ where $c_{\alpha_0+\alpha_1+\alpha_2+1}=\lambda_3-1=8$.  

\begin{figure}[H]
\includegraphics[width=0.4\textwidth]{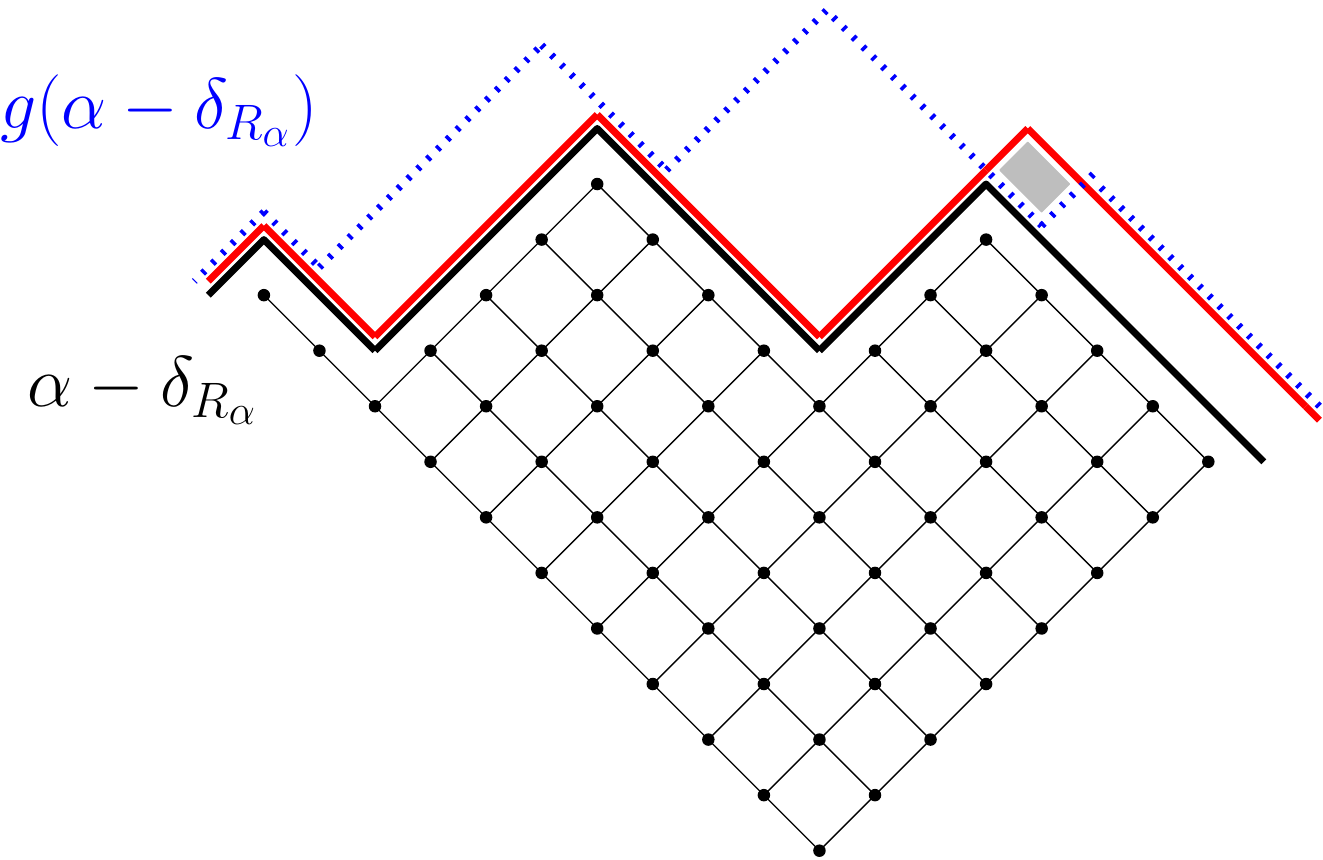}
\caption{An illustration of the function $g(\alpha)$}\label{guncgunc}
\end{figure}

Proposition~\ref{proj-int} and~\ref{inj-int} in combination with Proposition~\ref{ir} show that Auslander-Reiten translation sends $[[f(\alpha),\alpha]]$ to $\pm [[\alpha-\delta_{R_{\alpha}},g(\alpha-\delta_{R_{\alpha}})]]$. As we have seen, the function $\alpha\mapsto g(\alpha-\delta_{R_{\alpha}})$ is important, and it will be useful to calculate it more directly. So, we will define a new function $\tilde{f}$, and then later prove that $\tilde{f}(\alpha)=g(\alpha-\delta_{R_{\alpha}})$ in Lemma~\ref{tilde-fog}.

We define the function $\tilde{f}$ from $E_{L}$ to $E_{L}$ as follows: Let $\alpha\in E_{L}$. First apply $\forget$, then deduct one from the first occurrence of each $\lambda_i$, $i\in R_{\alpha}$ while maximizing all of the other indices, i.e. make the partition as large as possible. Then we fix all of the $0$'s, i.e. we put the first bar at the end of the $0's$ in the result. If we have $n$ fixed $k$ times in $\alpha$, we make $n$ fixed $k-1$ times in $\tilde{f}(\alpha)$. If we do not have any fixed $n$'s in $\alpha$, then we fix all $n$'s in $\tilde{f}(\alpha)$.

Formally,
\begin{align*}
\tilde{f}((0^{\alpha_0}|\lambda_1^{\alpha_1},&\lambda_2^{\alpha_2},\cdots,\lambda_r^{\alpha_r}|n^{\alpha_{r+1}}))\\
=&\begin{cases}
(0^{\alpha_0+1}|(\lambda_2-1)^{\alpha_1},\cdots,(\lambda_r-1)^{\alpha_{r-1}}|n^{\alpha_{r}-1})&\text{ if } \alpha_{r+1}=0 \text{ and } \lambda_1=1\\
(0^{\alpha_0+1}|(\lambda_2-1)^{\alpha_1},\cdots,(\lambda_r-1)^{\alpha_{r-1}},n^{\alpha_{r}}|n^{\alpha_{r+1}-1})&\text{ if }\alpha_{r+1}  \neq 0 \text{ and } \lambda_1=1\\
(|(\lambda_1-1)^{\alpha_0+1},(\lambda_2-1)^{\alpha_1},\cdots,(\lambda_r-1)^{\alpha_{r-1}}|n^{\alpha_{r}-1})&\text{ if } \alpha_{r+1}=0 \text{ and } \lambda_1\neq 1\\
(|(\lambda_1-1)^{\alpha_0+1},(\lambda_2-1)^{\alpha_1},\cdots,(\lambda_r-1)^{\alpha_{r-1}},n^{\alpha_{r}}|n^{\alpha_{r+1}-1})&\text{ if } \alpha_{r+1}\neq 0 \text{ and } \lambda_1\neq 1.
\end{cases} 
\end{align*}

\begin{example} Consider the same $\alpha$ as in Example~\ref{functions-on-enhanced}. Then
$$\tilde{f}(\alpha)=\tilde{f}((0^2|2^3,5,6^4,9^2|13^2))=(1^3,4^3,5,8^4,13^2|13)$$
\end{example}

\begin{lemma}~\label{tilde-fog}
We have $\tilde{f}(\alpha)=g(\alpha-\delta_{R_{\alpha}})$ for every $\alpha\in E_L$. 
\end{lemma}

\begin{proof} Let $\alpha=(0^{\alpha_0}|\lambda_1^{\alpha_1},\lambda_2^{\alpha_2},\cdots,\lambda_r^{\alpha_r}|n^{\alpha_{r+1}})\in E_L$.
Recall that in Remark~\ref{alpha-delta} we explained how we get the enhanced partition $\alpha-\delta_{R_{\alpha}}$. So, we have the $0$'s fixed in $\alpha-\delta_{R_{\alpha}}$ only when $\lambda_1=1$ in $\alpha$. Recall also that since $\alpha\in E_L$, we have $\lambda_1\neq 0$. Firstly assume $\lambda_1\neq 1$, i.e. there is no $0$ fixed in $\alpha-\delta_{R_{\alpha}}$.  Then we get $\alpha-\delta_{R_{\alpha}}=(|0^{\alpha_0},(\lambda_1-1)^{\alpha_1},(\lambda_2-1)^{\alpha_2},\cdots,(\lambda_r-1)^{\alpha_r}|n^{\alpha_{r+1}})$. Now, we apply the map $g$. We get the following which is the desired result. 
\begin{align*}
g(\alpha-\delta_{R_{\alpha}})=
\begin{cases}
(|(\lambda_1-1)^{\alpha_0+1},(\lambda_2-1)^{\alpha_1},\cdots,(\lambda_r-1)^{\alpha_{r-1}}|n^{\alpha_r-1})  &\text{ if } \alpha_{r+1} = 0\\
(|(\lambda_1-1)^{\alpha_0+1},(\lambda_2-1)^{\alpha_1},\cdots,(\lambda_r-1)^{\alpha_{r-1}},n^{\alpha_r}|n^{\alpha_{r+1}-1}) &\text{ if } \alpha_{r+1}\neq 0.
\end{cases}
\end{align*}

The case $\lambda_1=1$ can be calculated similarly. This finishes the proof.
\end{proof}

\begin{proposition} $\tau([\mathcal{P}_{\alpha}])=\pm [\mathcal{P}_{g(\alpha-\delta_{R_{\alpha}})}]$, or equivalently                                           $\tau([\mathcal{P}_{\alpha}])=\pm [\mathcal{P}_{\tilde{f}(\alpha)}]$.
\end{proposition}

\begin{proof} Since we proved that $f$ and $g$ are inverses of each other in Lemma~\ref{fog}, it is easy to see that the class of the projective resolution $\mathcal{P}_{g(\alpha-\delta_{R_{\alpha}})}$ of the enhanced partition $g(\alpha-\delta_{R_{\alpha}})$ in $\mathcal{K}_0$ is $\pm [[\alpha-\delta_{R_{\alpha}},g(\alpha-\delta_{R_{\alpha}})]]$.
\end{proof}

To sum up, for any enhanced partition $\alpha \in E_L$, one can write the projective resolution $\mathcal{P}_{\alpha}$ and find the class of $\mathcal{P}_{\alpha}$ in $\mathcal{K}_0$ which is $[[f(\alpha),\alpha]]$. The class of $\tau \mathcal{P}_{\alpha}$ in $\mathcal{K}_0$ is $[[\alpha-\delta_{R_{\alpha}},g(\alpha-\delta_{R_{\alpha}})]]$. Moreover, we can determine $g(\alpha-\delta_{R_{\alpha}})$ directly from $\alpha$ by using the map $\tilde{f}$. We are now in a good position to iterate the application of $\tau$.

\section{Configurations and Enhanced Partitions}
The goal of this section is to define a bijection from $E_L$ to 
a set $\mathcal{D}_{m,n}$ which we define below and has a natural action of $\mathbb{Z}/m+n+1$. 

\subsection{A combinatorial model: Configurations}

Consider $\mathcal{Z}:=\{-m,\cdots,-1,0,1,\cdots,n\}$ for the elements of $\mathbb{Z}/m+n+1$. 

\begin{definition} \emph{A configuration} $D$ is an increasing sequence of $m$ elements from $\mathcal{Z}$. We write $D=\{i_1< i_2<\cdots<i_m\}$ for a configuration. The set $\mathcal{D}_{m,n}$ denotes all configurations. Notice that the cardinality of $\mathcal{D}_{m,n}$ is $\binom{m+n+1}{m}$.
\end{definition}

Consider a configuration $D$. By $D\{i\}$ we mean the $i$ times shifted version of $D$, i.e. $D\{i\}$ is the set of elements $\{i_1-i,\cdots, i_m-i\}$ which are sorted into increasing order, and we write $sorted\{i_1-i,\cdots, i_m-i\}$. Call this operation $\{i\}$ a shift. Clearly, $D\{m+n+1\}=D$. The set of all configurations $D\{i\}$ for all $0\leq i \leq m+n$ is called \emph{the full orbit of} $D$.

Recall that $E_{L}$ is the set of enhanced partitions of the form $\alpha=(0^{\alpha_0}|\lambda_1^{\alpha_1},\lambda_2^{\alpha_2},\cdots,\lambda_r^{\alpha_r}|n^{\alpha_{r+1}})$ where $\lambda_1\neq 0$. We also write it as a sequence $\alpha=(a_1,a_2,\cdots,a_m)$. We are going to define a function $\psi$ from $E_L$ to $\mathcal{D}_{m,n}$ as follows.

Let us first define a function $\mu_{\alpha}: \{1,\cdots,m\}\to \mathcal{Z}$ 
\[ \mu_{\alpha}(j) =
  \begin{cases}
    a_j  & \quad \text{if } j=\sum_{i=0}^k \alpha_i \text{ for some } k\in\{0,\ldots,r\},\\
    -j  & \quad \text{otherwise.}\\
  \end{cases}
\] It is easy to see that it is well-defined. Now, we define $$\psi: E_L\to \mathcal{D}_{m,n}$$ as follows: $$\psi(\alpha)=sorted\{\mu_{\alpha}(j)\mid \text{ for } j\in \{1,\cdots,m\}\}.$$

\begin{example}\label{ex-frame} We continue in the setting of Example~\ref{pr-re}. Consider the enhanced partition $\alpha=(0|2,2,3,7|)$, then the corresponding configuration is $\psi(\alpha)=\{-2<0<2<3<7\}.$ Now, consider the enhanced partition $\alpha'=(0|2,2,3|7)$. Then the corresponding configuration is $\psi(\alpha')=\{-5<-2<0<2<3\}.$
\end{example}

\begin{lemma}~\label{psi-bijection} The map $\psi$ is a bijection.
\end{lemma}

\begin{proof} We can think of an element in $E_L$ as a multiset on $\{0,1,\cdots,n,n^*\}$ where we use $n^*$ to represent the fixed $n$'s. Notice that the cardinalities of $E_L$ and $\mathcal{D}_{m,n}$ are the same. Then, it is enough to show that the map $\psi$ is injective. 

Let $\alpha=(0^{\alpha_0}|\lambda_1^{\alpha_1},\lambda_2^{\alpha_2},\cdots,\lambda_r^{\alpha_r}|n^{\alpha_{r+1}})$, $\beta=(0^{\beta_0}|\xi_1^{\beta_1},\xi_2^{\beta_2},\cdots,\xi_s^{\beta_s}|n^{\beta_{s+1}})\in E_L$. To simplify the exposition, we assume $\alpha_i>1$ for all $i\in\{0,1,\cdots,r\}$ and $\beta_j>1$ for all $j\in\{0,1,\cdots,s\}$. The general case is similar. 

We also write 
\begin{align*}
\alpha =(a_1,\ldots,a_{\alpha_0-1},a_{\alpha_0}|a_{\alpha_0+1},\ldots &, a_{\alpha_0+\alpha_1-1},a_{\alpha_0+\alpha_1},\ldots, \\
& a_{\alpha_0+\alpha_1+\cdots+\alpha_{r-1}+1},\ldots,a_{\alpha_0+\alpha_1+\cdots+\alpha_r-1},a_{\alpha_0+\alpha_1+\cdots+\alpha_r}|n^{\alpha_{r+1}})\\
\beta =(b_1,\ldots,b_{\beta_0-1},b_{\beta_0},b_{\beta_0+1},\ldots &, b_{\beta_0+\beta_1-1},b_{\beta_0+\beta_1},\ldots, \\
& b_{\beta_0+\beta_1+\cdots+\beta_{s-1}+1},\ldots,b_{\beta_0+\beta_1+\cdots+\beta_s-1},b_{\beta_0+\beta_1+\cdots+\beta_s}|n^{\beta_{s+1}})
\end{align*}

Assume $\psi(\alpha)=\psi(\beta)$, then we have

\begin{align*}
\{\underbrace{-m<-m+1<\cdots<-m+\alpha_{r+1}-1}_{\alpha_{r+1}}<\underbrace{-x_0<-x_0+1<\cdots<-x_0+\alpha_r-2}_{\alpha_r-1}<\\
\cdots<\underbrace{-x_r<-x_r+1<\cdots<-x_r+\alpha_0-2}_{\alpha_0-1}<0<\lambda_1<\cdots<\lambda_r\}\\
\\
=\{\underbrace{-m<-m+1<\cdots<-m+\beta_{s+1}-1}_{\beta_{s+1}}<\underbrace{-y_0<-y_0+1<\cdots<-y_0+\beta_s-2}_{\beta_s-1}<\\
\cdots<\underbrace{-y_s<-y_s+1<\cdots<-y_s+\beta_0-2}_{\beta_0-1}<0<\xi_1<\cdots<\xi_s\}
\end{align*} where $x_k=\alpha_0+\alpha_1+\cdots+\alpha_{r-k}-1$ for $k\in\{0,1,\cdots,r\}$ and $y_k=\beta_0+\beta_1+\cdots+\beta_{s-k}-1$ for $k\in\{0,1,\cdots,s\}$.

Since $\lambda_i$'s and $\xi_i$'s are all positive and linearly ordered, then for each $i> 0$, $\lambda_i=\xi_i$ and $r=s$. Now, assume $\alpha_{r+1}\neq\beta_{s+1}$. Without loss of generality, say $\beta_{s+1}<\alpha_{r+1}$. Then $-m+\beta_{r+1}-1< -m+\alpha_{r+1}-1$ which implies $-m+\beta_{r+1}\leq -m+\alpha_{r+1}-1$. Thus, $-m+\beta_{r+1}\in \psi(\alpha)$. But this is a contradiction, because $-m+\beta_{r+1}$ cannot be in $\psi(\beta)$. By the same argument, we can prove that for each $j$, $\alpha_j=\beta_j$. This proves that $\alpha=\beta$. 
\end{proof}

Now, let $\mathcal{F}:=\{0,1,\cdots,n,n^*\}$ where $n^*$ is a formal element distinct from $n$. We also define a map $\varphi$ from $\mathcal{D}_{m,n}$ to $E_L$. Let $D=\{i_1< i_2<\cdots<i_m\}$. First of all, let us define a function $\eta_D:\{1,\cdots,m\}\to \mathcal{F}$ as follows. For each $j$,
\begin{equation}~\label{eta}
\eta_D(j) =
  \begin{cases}
    i_{|i_j|+j}  & \quad \text{if } i_j< 0 \text{ and } |i_j|+j< m+1,\\
    n^* & \quad \text{if } i_j< 0 \text{ and } |i_j|+j= m+1,\\
    i_j  & \quad \text{otherwise.}
  \end{cases}
\end{equation}  

\begin{lemma} The map $\eta_D$ is well-defined. 
\end{lemma}
\begin{proof} Let $D=\{i_1< i_2<\cdots<i_m\}$. The only case we need to check is that if $|i_j|+j< m+1$ and $i_j< 0$, then $i_{|i_j|+j}\in \mathcal{F}$. Assume $i_l,\cdots,i_m$ are nonnegative. Then we have the following,
$$m> |i_1|> |i_2| > |i_3| > \cdots > |i_{l-1}| > 0$$ which implies
$$m\geq |i_1|+1\geq |i_2|+2 \geq |i_3|+3 \geq \cdots \geq |i_{l-1}| + l-1 \geq l.$$ Therefore, for $1\leq j\leq l-1$, we have $m\geq |i_j|+j\geq l$. This shows that $i_{|i_j|+j}\in \mathcal{F}.$
\end{proof}

Now, we define a map \[\varphi: \mathcal{D}_{m,n}\to E_L\] as follows: \[\varphi(D)=enhanced(sorted\{\eta_D(j)\mid j\in \{1,\cdots,m\}\})\] where we first sort into a non-decreasing order with $n^*$ after $n$ and then enhance the result as follows: fix all of the $0$'s and fix all of the $n^*$'s. This map is the inverse map of $\psi$. But we will not prove it since we do not need this fact.

\section{The main result}  We consider the Grothendieck group of the bounded derived category of the incidence algebra $\mathcal{A}$ of the poset of order ideals in a grid $\mathsf{J(P_{m,n})}$.  In this section, we are going to prove the following:

\begin{theorem}~\label{tau-on-grid}
The Auslander-Reiten translation $\tau$ has finite order on $\mathcal{K}_0(\mathcal{D}^b(\mathcal{A}))$ for the incidence algebra $\mathcal{A}$ of the poset of order ideals $\mathsf{J(P_{m,n})}$ of a grid poset $\mathsf{P_{m,n}}$. Specifically, $\tau^{m+n+1}=\pm id$. 
\end{theorem}

Our main result follows from two auxiliary Propositions.

\begin{enumerate}[(1)]
\item In Proposition \ref{tau-finite} we show that Auslander-Reiten translation satisfies $\tau^{2(m+n+1)} = id$ on the elements $[\mathcal{P}_{\alpha}]$ in $\mathcal{K}_0$.
\item In Proposition \ref{generate} we show that the Grothendieck group $\mathcal{K}_0$ is generated by the elements of the form $[\mathcal{P}_{\alpha}]$. 
\end{enumerate} 

\begin{proposition}\label{tau-finite} $\tau^{2(m+n+1)}=id$ on the elements $[\mathcal{P}_{\alpha}]$ in $\mathcal{K}_0$. 
\end{proposition}

\begin{proof} 

First, we are going to prove that the following diagram commutes since the shift $\{1\}$ has finite order of $(m+n+1)$ on $\mathcal{D}_{m,n}$. 
\begin{center}
$\xymatrix{E_L \ar@{>}[r]^{\psi} \ar@{>}[d]_{\tilde{f}}  & \mathcal{D}_{m,n} \ar@{>}[d]^{\{1\}}\\
           E_L\ar@{>}[r]_{\psi} & \mathcal{D}_{m,n}}$ 
\end{center}
Let $\alpha=(0^{\alpha_0}|\lambda_1^{\alpha_1},\lambda_2^{\alpha_2},\cdots,\lambda_r^{\alpha_r}|n^{\alpha_{r+1}})\in E_L$ be an enhanced partition. To simplify the exposition, assume $\alpha_i>1$ for all $i$ and assume $\lambda_1\neq 1$. The general case is similar. Then, $\tilde{f}(\alpha)=((\lambda_1-1)^{\alpha_0+1},(\lambda_2-1)^{\alpha_1},\cdots,(\lambda_r-1)^{\alpha_{r-1}},n^{\alpha_{r}}|n^{\alpha_{r+1}-1})$

\begin{align*}
\psi(\alpha)&=\{\underbrace{-m<-m+1<\cdots<-m+\alpha_{r+1}-1}_{\alpha_{r+1}}<\underbrace{-x_0<-x_0+1<\cdots<-x_0+\alpha_r-2}_{\alpha_r-1}<\\
&\cdots<\underbrace{-x_r<-x_r+1<\cdots<-x_r+\alpha_0-2}_{\alpha_0-1}<0<\lambda_1<\cdots<\lambda_r\}\\ 
&\\
\psi(\tilde{f}(\alpha))&=\{\underbrace{-m<-m+1<\cdots<-m+\alpha_{r+1}-2}_{\alpha_{r+1}-1}<\underbrace{-x_0-1<-x_0<\cdots<-x_0+\alpha_r-3}_{\alpha_r-1}\\
&<\cdots<\underbrace{-x_{r-1}-1<-x_{r-1}<\cdots<-x_{r-1}+\alpha_r-3}_{\alpha_{1}-1}\\
&<\underbrace{-x_r-1<-x_r<\cdots<-x_r+\alpha_0-2}_{\alpha_0}<\lambda_1-1<\cdots<\lambda_r-1<n\}
\end{align*} where the $x_i$'s are defined as in the proof of Lemma~\ref{psi-bijection}. 

From these calculations, it is easy to see that $\psi(\alpha)\{1\}=\psi(\tilde{f}(\alpha))$. This shows that $\tau^{m+n+1}[\mathcal{P}_{\alpha}]=\pm [\mathcal{P}_{\alpha}]$. Therefore, $\tau^{2(m+n+1)}[\mathcal{P}_{\alpha}]= [\mathcal{P}_{\alpha}]$. The proof of the case $\lambda_1=1$ is similar. Also, observe that the order of $\tau$ cannot be less than $(m+n+1)$ because this is obviously true for the action of $\{1\}$ on $\mathcal{D}_{m,n}$. This finishes the proof. 
\end{proof}

\begin{example} In this example, we will write the action of $\tau$ algebraically and combinatorially. Assume $m=5$ and $n=3$.

Let $\alpha=(|1,1,2,3,3|)$. Write the projective resolution as follows:
\begin{align*}
\mathcal{P}_{\alpha}:0\rightarrow P^{-3}_{(0,0,1,2,2)}\rightarrow P^{-2}_{(0,0,1,3,3)} & \oplus P^{-2}_{(1,1,1,2,2)}\oplus P^{-2}_{(0,0,2,2,2)}\rightarrow \\
&P^{-1}_{(0,0,2,3,3)}\oplus P^{-1}_{(1,1,1,3,3)}\oplus P^{-1}_{(1,1,2,2,2)}\rightarrow P^0_{(|1,1,2,3,3|)}\rightarrow 0
\end{align*}
with the homology $[[f(\alpha),\alpha]]=[[(0,1,2,2,3),(1,1,2,3,3)]].$

Apply $\tau$ to $\mathcal{P}_{\alpha}$:
\begin{align*}
\mathcal{I}_{\alpha}:0\rightarrow I^{-3}_{(0,0,1,2,2)}\rightarrow I^{-2}_{(0,0,1,3,3)} & \oplus I^{-2}_{(1,1,1,2,2)}\oplus I^{-2}_{(0,0,2,2,2)}\rightarrow \\
&I^{-1}_{(0,0,2,3,3)}\oplus I^{-1}_{(1,1,1,3,3)}\oplus I^{-1}_{(1,1,2,2,2)}\rightarrow I^0_{(1,1,2,3,3)}\rightarrow 0
\end{align*}
with the homology $[[(0,0,1,2,2),(0,1,1,2,3)]].$

Note that $\tilde{f}((|1,1,2,3,3|))=(0|1,1,2|3)$. So, $\tau \mathcal{P}_{(|1,1,2,3,3|)}\cong \mathcal{P}_{(0|1,1,2|3)}[-2]$. In the Grothendieck group, we have the following

$$\tau [[(0,1,2,2,3),(1,1,2,3,3)]]=[[(0,0,1,2,2),(0,1,1,2,3)]]$$ Now, let us look at the action of $\tau$ combinatorially. Firstly, we find the corresponding configuration $D$ of $\alpha$ which is $D=\psi(\alpha)=\{-4\leq -1\leq 1\leq 2\leq 3\}$. 

We now compute that $\psi\tilde{f}(\alpha)=\psi((0|1,1,2|3))=\{-5<-2<0<1<2\}$ which equals $D\{1\}$. 
\end{example}

\begin{proposition}~\label{sign} Let $\mathsf{J(P_{m,n})}$ be the poset of order ideal of a grid poset $\mathsf{P_{m,n}}$. If $m$, $n$ are both even, then the order of $\tau$ is $2(m+n+1)$. Otherwise, the order of $\tau$ is $m+n+1$.
\end{proposition}

\begin{proof} 
First, we observe that if $|R_{\alpha}|$ is odd, then $[\tau\mathcal{P}_{\alpha}]$ is a positive sum of simples in $\mathcal{K}_0$; if $|R_{\alpha}|$ is even, then $[\tau\mathcal{P}_{\alpha}]$ is a negative sum of simples in $\mathcal{K}_0$.

Let us state this fact in terms of configurations as follows. We will work with sign configurations which are just configurations with a sign attached. Let $D$ be the corresponding configuration to $\alpha$ with a sign attached. In this case, $|R_{\alpha}|$ is the number of positive entries in $D$. We define the action of $\{1\}$ on signed configurations as follows: if $|R_{\alpha}|$ is odd, then $D\{1\}$ will have the same sign; if $|R_{\alpha}|$ is even, then $D\{1\}$ will have the opposite sign. We will write this fact in an explicit way as follows. Let $D=\{a_1<\cdots<a_m\}$ for $\alpha=(a_1,\cdots,a_m)$.

We first define $r:D\to \{1,-1\}$ by sending $a_i\mapsto r(a_i)$ where 
\[r(a_i)=\begin{cases} -1 &\text{ if \ \ \ \ \ } 0<a_i\leq n\\
1 &\text{ if } -m\leq a_i\leq 0.
\end{cases}\]

Then $\si(D\{1\})$ is as follows:

\[\si(D\{1\})=(-1)\left(\prod^m_{j=1}r(a_j)\right)\si(D)\]

We know that each $a_i$ will be non-negative $n$ times in the full orbit of $D$. So, we have \[\displaystyle\prod^{m+n}_{j=0}r(a_i-j)=(-1)^n\]

Now, let us determine the sign of $D\{m+n+1\}$.\\

$\si(D\{m+n+1\})$
\begin{align*}
&=\left((-1)\prod^{m}_{j=1}r(a_j-m-n)\right)\si(D\{m+n\})\\
&=\left((-1)\prod^{m}_{j=1}r(a_j-m-n)\right)\left((-1)\prod^{m}_{j=1}r(a_j-m-n-1)\right)\cdots\left((-1)\prod^{m}_{j=1}r(a_j)\right)\si(D)\\
&=(-1)^{m+n+1}\left(\prod^{m+n}_{j=0}r(a_1-j)\right)\cdots\left(\prod^{m+n}_{j=0}r(a_m-j)\right)\si(D)\\
&=(-1)^{m+n+1}(-1)^{nm}\si(D)=(-1)^{(m+1)(n+1)}\si(D)
\end{align*}
From this, we see that when $m$ and $n$ are both even, the order is $2(m+n+1)$. Otherwise, it is $m+n+1$.
\end{proof}

\begin{proposition}\label{generate} The elements $[\mathcal{P}_{\alpha}]$, $\alpha$ is a partition, generates the Grothendieck group $\mathcal{K}_0$ of the incidence algebra $\mathcal{A}$ of the poset $\mathsf{J(P_{m,n})}$.
\end{proposition}

\begin{proof} The Grothendieck group $\mathcal{K}_0$ is generated by all the isomorphism classes of indecomposable projective modules $[P_{\alpha}]$, $\alpha\in \mathsf{J(P_{m,n})}$. Now, we will think of $\alpha$ as an enhanced partition with the first bar is placed after $0$'s and the second bar at the very end. Let us define $L_{x}=[[f(x),x]]$ where $x$ is an enhanced partition. We will show that each $[P_{\alpha}]$ can be written as a linear combination of elements of the form $[L_{x}]$. We will proceed by strong induction on partitions ordered lexicographically. The base case is $\alpha=(0,\cdots,0)$. Then $[P_{\alpha}]=[L_\alpha]$, and we are done.

Recall that we get the element $[\mathcal{P}_{\alpha}]$ in $\mathcal{K}_0$ by taking the Euler characteristic of the projective resolution $\mathcal{P}_{\alpha}$. 

Recall also the notation from Subsection~\ref{s-proj-res}. Let us write each partition $\alpha-\delta_J$ where $J\subseteq R_{\alpha}$ and $J\neq \emptyset$. Notice that each $\alpha-\delta_J$ comes before $\alpha$ in the lexicographical order. 

Now, we write $[P_{\alpha}]=[\mathcal{P}_{\alpha}]+[\displaystyle\bigoplus_{\substack{J\subseteq R_{\alpha},\\|J|=1}} P_{\alpha-\delta_J}]-[\bigoplus_{\substack{J\subseteq R_{\alpha},\\|J|=2}} P_{\alpha-\delta_J}]+\cdots+(-1)^{|J|} [P_{\alpha-\delta_{R_{\alpha}}}]$. Therefore, by the induction hypothesis, each $[P_{\alpha-\delta_J}]$ can be written as a linear combination of elements of the form $[L_{x}]$. So, we have $$\displaystyle\sum_{\substack{J\neq\emptyset}}(-1)^{(|J|+1)}[ P_{\alpha-\delta_J}]=\sum_{x}a_{x} [L_{x}]$$ And we know that $[\mathcal{P}_{\alpha}]=[L_{\alpha}]$. Now, we have the desired result: $$[P_{\alpha}]=[L_{\alpha}]+\sum_{x}a_{x} [L_{x}]$$ 
\end{proof}


\section{A general framework : Cominuscule Posets}\label{cominuscule-section} In this section, we will investigate the action of $\tau$ for the poset of order ideals of cominuscule posets.

\subsection{Root Posets and Cominuscule Posets} 

Let $\Phi$ be a finite root system, and $\Delta \subset \Phi$ be the set of simple roots. Let $\Phi^+ \subset \Phi$ be the subset that consists of the positive roots, which are the nonnegative linear combinations of elements in $\Delta$. One can define a partial order on $\Phi^+$ naturally as follows: For $\sigma,\ \sigma'\in \Phi^+$, we say that $\sigma'\prec\sigma$ if and only if $\sigma-\sigma'$ is a sum of positive roots. This poset is called the poset of positive roots of the root system $\Phi$. 

The height of a root $\sigma$ is defined as the sum of the coefficients in the expression of $\sigma$ as a linear combination of simple roots. The poset $\Phi^{+}$ has always a highest root, say $\eta$. For a detailed discussion of the subject, see \cite[Chapter 3]{Humphreys78}.

\begin{example}\label{rootposet} The following is the Hasse diagram of the poset of positive roots of $A_3$. The set of simple roots is $\Delta=\{\sigma_1,\sigma_2,\sigma_3\}$ and $\eta=\sigma_1+\sigma_2+\sigma_3$ is the highest root.
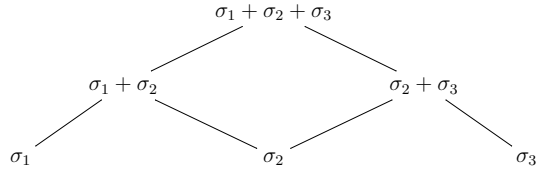
\begin{figure}[H]
\scalebox{0.7}{
$$    \xymatrix{
& & & {\sigma_1+\sigma_2+\sigma_3} \ar@{-}[dr] & &  & & \\
     & & {\sigma_1+\sigma_2} \ar@{-}[dr] \ar@{-}[ur] & & {\sigma_2+\sigma_3} \ar@{-}[dr]  & & & \\
  &{\sigma_1} \ar@{-}[ur]&  & {\sigma_2} \ar@{-}[ur] & & {\sigma_3}  &  & &  }$$}
\caption{Root poset of $A_3$}
\end{figure}
\end{example}

\begin{definition}~\label{cominuscule} Let $\eta$ be the highest root. A simple root $\sigma$ is called a cominuscule root if the multiplicity of $\sigma$ in the simple root expansion of $\eta$ is $1$.
\end{definition} 

\begin{definition} Given a root system $\Phi$, an interval $\mathsf{C}$ of the form $[\sigma,\eta]$ in the root poset for $\Phi^+$ where $\sigma$ is cominuscule root and $\eta$ is the highest root is called a \emph{cominuscule poset}. 
\end{definition}

Cominuscule posets appear in representation theory of Lie groups, Schubert calculus, and combinatorics. For more details about root systems see \cite[Chapter 3]{Humphreys78}, for cominuscule posets \cite[Section 9.0]{BL00}, \cite{RS13} and \cite{TY09}.

\begin{example} In Example \ref{rootposet}, all simple roots are cominuscule roots. The following are cominuscule posets for the simple roots $\sigma_1$ and $\sigma_2$, respectively.

\scalebox{0.7}{
$$ \qquad \qquad  \xymatrix{
  &&\sigma_1+\sigma_2+\sigma_3 \ar@{-}[dl]\\
  &\sigma_1+\sigma_2\ar@{-}[dl]&\\
 \sigma_1 &&  
 }
 \qquad \qquad
  \xymatrix{
  & \ar@{-}[dr]\ar@{-}[dl]\sigma_1+\sigma_2+\sigma_3 &\\
  \sigma_1+\sigma_2 \ar@{-}[dr]& &\ar@{-}[dl]\sigma_2+\sigma_3\\
  &\sigma_2 & 
 }$$}

\end{example}

The following shows an illustration of the possible shapes of cominuscule posets except the exceptional cases. The exceptional cominuscule posets will be discussed at the end of this section.
\begin{figure}[H]
\includegraphics[width=0.63\textwidth]{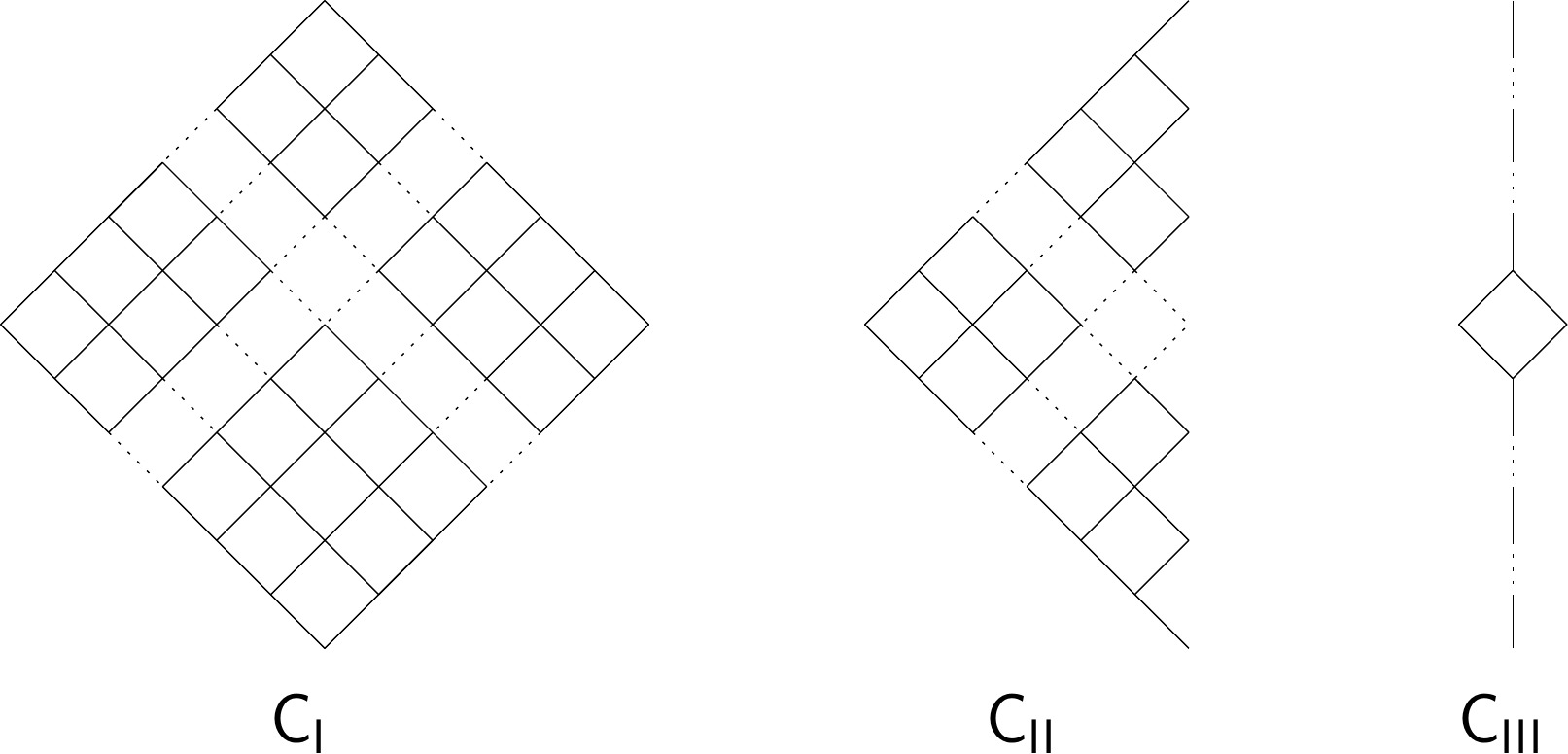}
\caption{Cominuscule posets of type $\mathsf{I}$, $\mathsf{II}$, $\mathsf{III}$}\label{set-cominuscule-posets}
\end{figure}

\begin{conjecture}~\label{conj}
Auslander-Reiten translation $\tau$ has finite order on the Grothendieck group of the bounded derived category for the incidence algebra of the poset $\mathsf{J(C)}$. Specifically, $\tau^{2(h+1)}=id$ where $h$ is the Coxeter number for the relevant root systems ($A$, $B$, $C$, $D$, $E_6$ or $E_7$). 
\end{conjecture}

In this paper, we prove that the Conjecture~\ref{conj} is true for cominuscule posets $\mathsf{J(C_I)}$, $\mathsf{J(C_{III})}$, $\mathsf{J(C_{E_6})}$, $\mathsf{J(C_{E_7})}$.

\begin{remark} If the same cominuscule poset comes from two different root systems, the resulting orders of the Coxeter transformation agree. So, Conjecture~\ref{conj} is consistent.
\end{remark}

We will proceed with a case by case analysis and prove some results.

\subsection{Type $A$ case}

All simple roots in the root poset of $A_n$ are cominuscule roots. So, they all give rise to cominuscule posets. 

The cominuscule poset over any simple root in $A_n$ is of type $\mathsf{I}$ in Figure~\ref{set-cominuscule-posets}. In other words, they are the grid posets of size $(n+1)$. We have the result that $\tau$ acts finitely on the poset of order ideals of a grid poset in Theorem~\ref{tau-on-grid}. Recall that the Coxeter number $h$ is $n+1$ in type $A_n$. Therefore, the Conjecture~\ref{conj} holds for the type $A_n$, i.e. $\tau^{2(h+1)}=id$ in $\mathcal{K}_0(\mathcal{D}^b(\mathcal{A}))$ for the incidence algebra $\mathcal{A}$ of the poset of order ideals $\mathsf{J(C_I)}$. 


\begin{example} The following figure is an illustration of root poset of $A_n$. The shaded area shows the cominuscule poset $\mathsf{C_I}$ over the simple root $\sigma_k$.

\begin{figure}[H]
\includegraphics[width=0.54\textwidth]{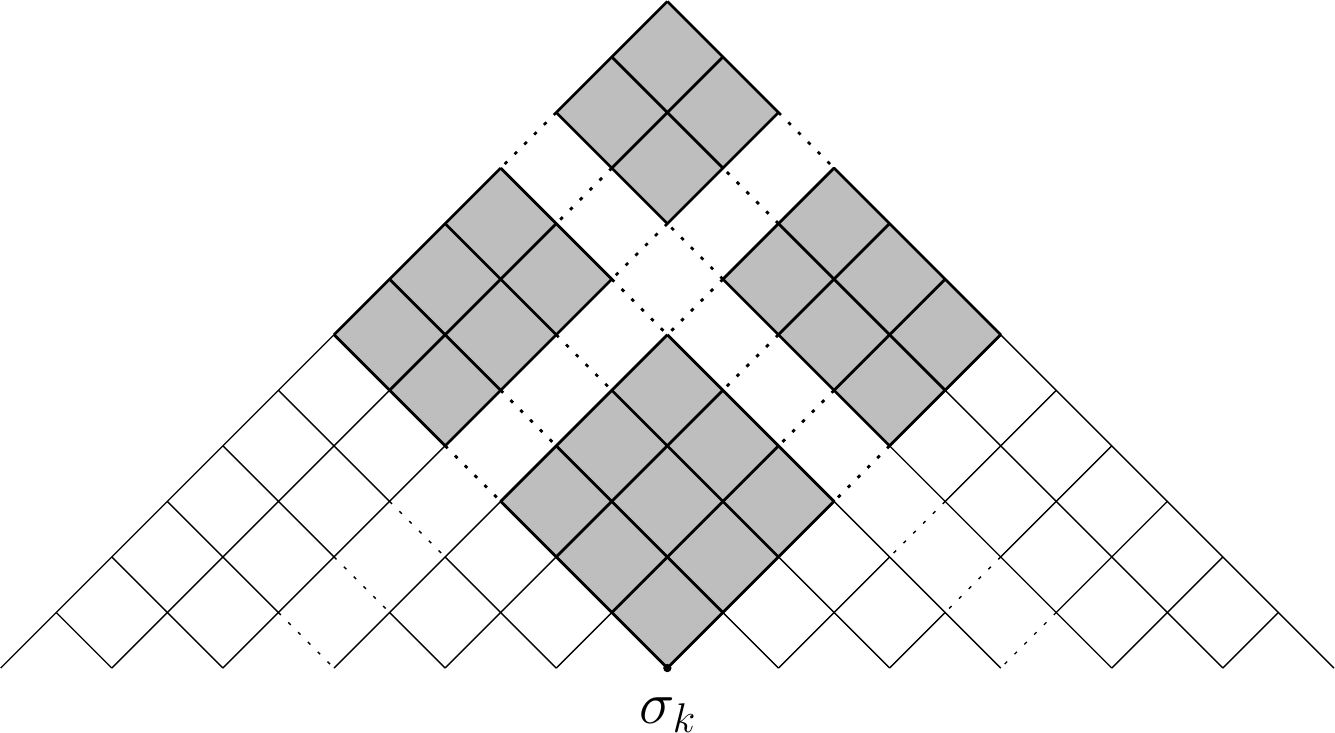}
\caption{The root poset of $A_n$ and the cominuscule poset $\mathsf{C_I}$ over the simple root $\sigma_k$.}
\end{figure}
\end{example}

\subsection{Type $B$ case} In type $B_n$, the only simple root which is a cominuscule root is $\sigma_1$ and the cominuscule poset is the grid poset $\mathsf{P_{1,2n-1}}$. In this case, the order of $\tau$ is $2n+1$ by Proposition~\ref{sign}. Therefore, we have the desired result, i.e. $\tau^{2(1+2n-1+1)}=\tau^{2(h+1)}=id$ where the Coxeter number $h$ is $2n$. 

\subsection{Type $C$ case} In type $C_n$, the only simple root which gives rise to the cominuscule poset is $\sigma_n$ and the cominuscule poset is the type of $\mathsf{II}$ in Figure~\ref{set-cominuscule-posets}. 

\begin{example} Here we illustrate the root poset of $C_n$ and the shaded area shows the cominuscule poset over the simple root $\sigma_n$.
\begin{figure}[H]
\includegraphics[width=0.3\textwidth]{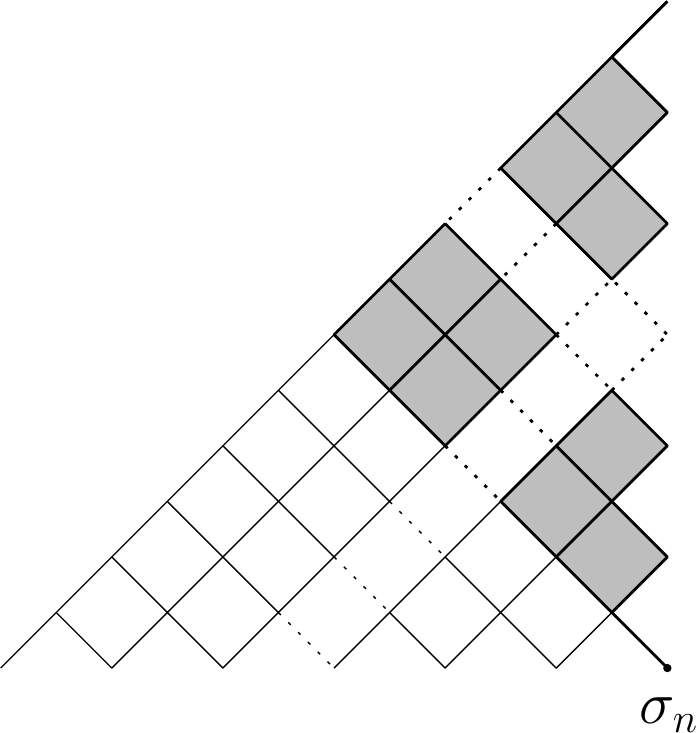}
\caption{The root poset of $C_n$ and the cominuscule poset $\mathsf{C_{II}}$ over the simple root $\sigma_n$.}
\end{figure}
This case is still open.
\end{example}

\subsection{The type D case} In type $D_n$, $n>3$, there are three simple roots $\sigma_1$, $\sigma_{n-1}$, and $\sigma_n$ which give rise to cominuscule posets. The simple roots $\sigma_{n-1}$ and $\sigma_n$ give rise to cominuscule posets of type $\mathsf{II}$ which is the same as the cominuscule poset for $C_n$. 

The simple root $\sigma_1$ gives rise to the cominuscule poset of type $\mathsf{III}$. The cominuscule poset $\mathsf{C_{III}}$ and the poset of order ideals $\mathsf{J(C_{III})}$ is shown as follows:

\begin{figure}[H]\label{dan}
$$\scalebox{0.5}{\xymatrix{
&2n-2\ar@{-}[d]&\\
&2n-3\ar@{-}[d]&\\
&\vdots&\\
&n+2\ar@{-}[d]&\\
&n+1\ar@{-}[dl]\ar@{-}[dr]&\\
n-1\ar@{-}[dr]&&n\ar@{-}[dl]\\
&n-2\ar@{-}[d]&\\
&n-3\ar@{-}[d]&\\
&\vdots&\\
&1\ar@{-}[u]&
}}\qquad \qquad \qquad
\scalebox{0.5}{\xymatrix{
&2n\ar@{-}[d]&\\
&2n-1\ar@{-}[d]&\\
&\vdots&\\
&n+3\ar@{-}[d]&\\
&n+2\ar@{-}[dl]\ar@{-}[dr]&\\
n\ar@{-}[dr]&&n+1\ar@{-}[dl]\\
&n-1\ar@{-}[d]&\\
&n-2\ar@{-}[d]&\\
&\vdots&\\
&1\ar@{-}[u]&
}}$$
\caption{The cominuscule poset $\mathsf{C_{III}}$ and the order ideal poset $\mathsf{J(C_{III})}$ for the simple root $\sigma_1$ in the root system for $D_n$.}
\end{figure}
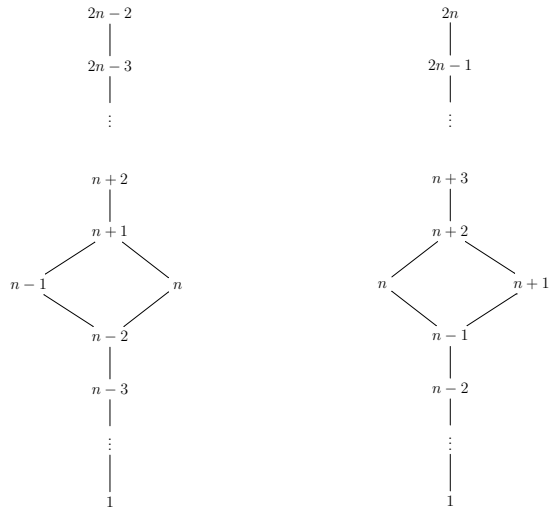

\begin{proposition}~\label{sefi} Let $\mathcal{A}_{(D_n,\sigma_1)}$ be the incidence algebra of the poset $\mathsf{J(C_{III})}$ of size $2n$ for the corresponding root system $D_n$. Then $\mathcal{D}^b(\mo\mathcal{A}_{(D_n,\sigma_1)})$ is fractionally Calabi-Yau.
\end{proposition}

\begin{proof}  Let $M_0$ be the module over the algebra $D_{2n}$ with dimension vector $(1,1,\cdots,1,1)$, and $M_i$ be the module with dimension vector $(1,1,\cdots,1,2,\cdots,2,1)$ where we have the dimension $2$ appears $i$ times for $i>0$. Let us now consider the following module: $T=P_1\oplus P_2\oplus\cdots\oplus P_{n-1}\oplus P_{2n-1}\oplus P_{2n}\oplus M_0\oplus M_1\oplus\cdots\oplus M_{n-1}$ which is illustrated in Figure~\ref{dn-tilting}.

\begin{figure}[H]
\includegraphics[width=0.7\textwidth]{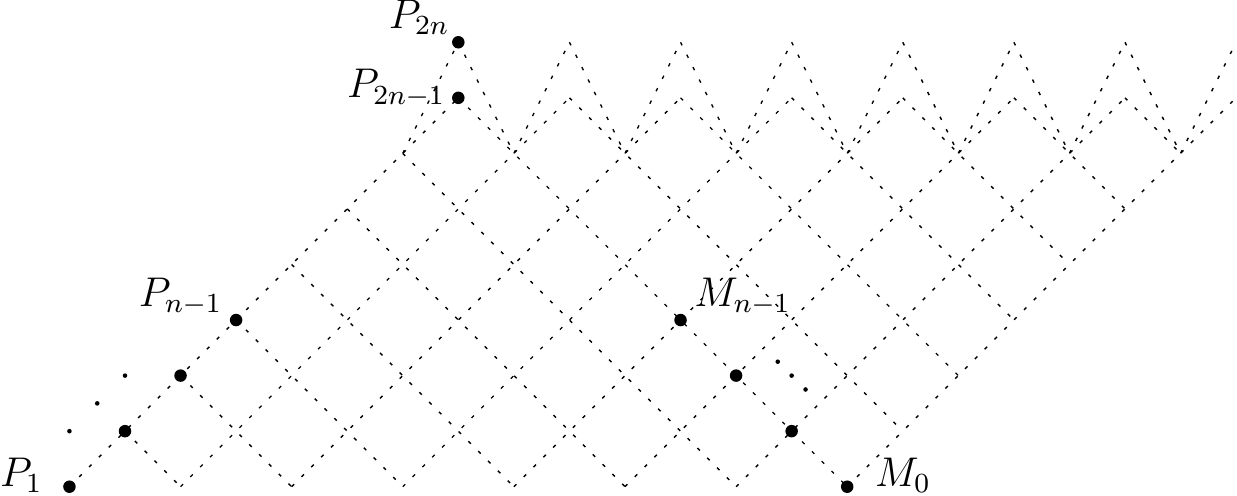}
\caption{An illustration of the module category of $D_{2n}$ and the tilting module $T$.}~\label{dn-tilting}
\end{figure}

It is not difficult to see $T$ is a tilting module. So, we consider the endomorphism algebra $\en T$ which is the algebra $\mathcal{A}_{(D_n,\sigma_1)}$. Therefore, we have that $\mathcal{D}^b(\mo\mathcal{A}_{(D_n,\sigma_1)})$ is derived equivalent to $\mathcal{D}^b(\mo D_{2n})$. We know by~\cite[Example 8.3(2)]{Keller05}, $\mathcal{D}^b(\mo D_{2n})$ is fractionally Calabi-Yau, and therefore $\mathcal{D}^b(\mo\mathcal{A}_{(D_n,\sigma_1)})$ is also fractionally Calabi-Yau. Thus, we have the desired result.
\end{proof}

\begin{remark}\label{referee} Proposition~\ref{sefi} can also be proved by using the technique of flip-flops of Ladkani~\cite[Theorem 1.1]{Ladkani07}. 
Let $\mathsf{P}$ be a finite poset, $\mathsf{P}^1$ be the poset with a unique maximum element added to $\mathsf{P}$, and $\mathsf{P}_0$ be the poset with a unique minimum element added to $\mathsf{P}$. Ladkani shows that $\mathsf{P}^1$ and $\mathsf{P}_0$ are derived equivalent. Using this fact, we can show that the poset $\mathsf{J(C_{III})}$ of size $2n$ is derived equivalent to $D_{2n}$.
\end{remark}

\begin{example} In this example, we use Ladkani's technique showing that $\mathsf{J(C_{III})}$ of size $8$ is derived equivalent to $D_{8}$.
\begin{figure}[H]
\includegraphics[width=0.7\textwidth]{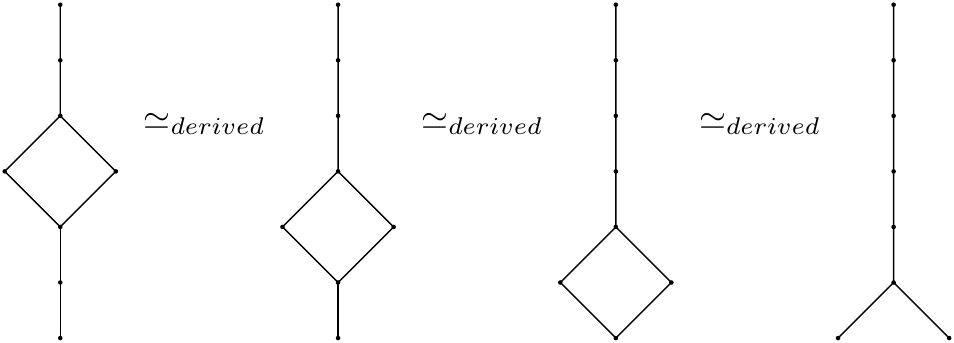}
\caption{The derived equivalent posets}
\end{figure}
\end{example}

\begin{corollary} The Auslander-Reiten translation $\tau$ on $\mathcal{K}_0(\mathcal{D}^b(\mo\mathcal{A}_{(D_n,\sigma_1)}))$ has finite order of $2(h+1)$ where $h$ is the Coxeter number for type $D_n$. 
\end{corollary}

\begin{proof} This is an immediate consequence of Proposition~\ref{sefi}. Since $\mathsf{J(C_{III})}$ of size $2n$ is derived equivalent to $D_{2n}$, we know that $\tau^{h_{2n}}=id$ in $\mathcal{K}_0$ where $h_{2n}=2(2n-1)$ is the Coxeter number for $D_{2n}$. To be consistent with our Conjecture~\ref{conj}, we write $\tau^{2(h+1)}=id$ in $\mathcal{K}_0$ where $h=2(n-1)$ is the Coxeter number for $D_{n}$.
\end{proof}

\subsection{Exceptional Cases} There are two cominuscule roots which give rise to the same cominuscule poset for type $E_6$ and there is only one cominuscule root which gives a cominuscule poset for type $E_7$.

\begin{figure}[H]
\includegraphics[width=0.44\textwidth]{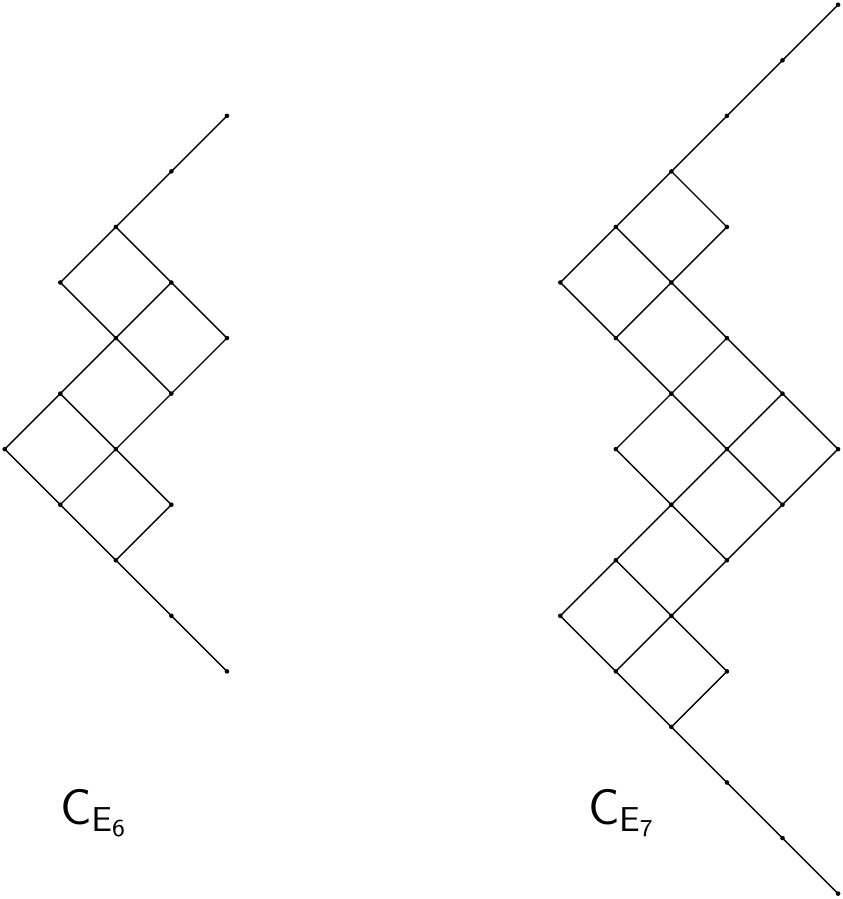}
\caption{The cominuscule poset $\mathsf{C_{E_6}}$ and the cominuscule poset $\mathsf{C_{E_7}}$}
\end{figure}

We checked that the result also holds for these two exceptional cases by using the mathematical software \emph{SageMath} \cite{sagemath}.

\bibliographystyle{plain}

\end{document}